\documentclass[11pt]{article}
\usepackage{amsmath}
\usepackage{amsfonts}
\usepackage{amssymb}
\usepackage{amsthm}
\usepackage{graphicx}
\usepackage{psfrag}
\usepackage{epsf}
\usepackage{mathrsfs,eucal,xspace}
\usepackage[all]{xy}
\usepackage{hyperref,amssymb,stmaryrd}
\usepackage{pstricks,pst-node}

\newtheorem{theorem}{Theorem}[section]
\newtheorem{proposition}[theorem]{Proposition}
\newtheorem{definition}[theorem]{Definition}
\newtheorem{lemma}[theorem]{Lemma}
\newtheorem{corollary}[theorem]{Corollary}
\newtheorem{remark}[theorem]{Remark}

\newcommand{\fraku}{\mathfrak{U}}

\newcommand{\A}{\mathcal{A}}

\newcommand{\spam}{\mathop{\mathrm{span}}}

\newcommand{\ints}{\mathbb{Z}}
\newcommand{\nats}{\mathbb{N}}
\newcommand{\reals}{\mathbb{R}}

\newcommand{\M}{\mathbb{M}}

\newcommand{\sphere}{\mathbb{S}^d}
\newcommand{\sph}{\mathbb{S}} 
\newcommand{\Exp}{\operatorname{Exp}}
\newcommand{\man}{\M}
\newcommand{\dif}{\mathrm{d}}

\newcommand{\Rad}{\left( \frac{h}{\Gamma_1 h_0}\right)}
\newcommand{\rad}{ \frac{h}{\Gamma_1 h_0}}
\renewcommand{\b}{\mathbf{b}}
\renewcommand{\a}{\mathbf{a}}
\newcommand{\bft}{\mathbf{t}}
\newcommand{\bfu}{\mathbf{u}}
\newcommand{\bfv}{\mathbf{v}}
\newcommand{\bfw}{\mathbf{w}}
\newcommand{\inj}{\mathrm{r}_\M}
\renewcommand{\d}{\mathrm{dist}}
\renewcommand{\aa}{\mathfrak{a}}
\newcommand{\comp}{\b^{\complement}}
\newcommand{\bb}{\mathfrak{b}}
\newcommand{\bbcomp}{\mathfrak{b}^{\complement}}
\renewcommand{\u}{\widetilde{\chi}}
\renewcommand{\A}{\mathcal{A}}
\newcommand{\ph}{\widetilde{\phi}}
\newcommand{\D}{\{1...d\}}
\newcommand{\sfp}{\mathsf{p}}
\newcommand{\sfq}{\mathsf{q}}
\newcommand{\stardom}{\mathcal D}  
\newcommand{\inrad}{r}  

\def\caln{{\mathcal N}}
\def\calc{{\mathcal C}}

\def\bfe{{\bf e}}
\def\hati{{\hat \imath}}
\def\hatj{{\hat \jmath}}
\def\bft{{\bf t}}
\def\bfu{{\bf u}}
\numberwithin{equation}{section}
\title{Kernel Approximation on Manifolds I: Bounding the Lebesgue Constant
\thanks{ \emph{2000 Mathematics
   Subject Classification:} 41A05, 41A63, 46E22, 46E35 }
\thanks{\emph{Key words:}
   manifold, Lebesgue constant, positive definite kernels, Sobolev spaces}}
\author{T.Hangelbroek\thanks{Texas A\&M
    University College Station, TX 77843, USA. Research supported
    by NSF Postdoctoral Research Fellowship.}, 
F. J. Narcowich\thanks{ Department of Mathematics, Texas A\&M
    University College Station, TX 77843, USA. Research
    supported by grants DMS-0504353 and DMS-0807033 from the National
    Science Foundation.}, 
J. D. Ward\thanks{ Department of Mathematics, Texas A\&M University
    College Station, TX 77843, USA. Research supported by
    grants DMS-0504353 and DMS-0807033 from the National Science
    Foundation.}  }

\begin{document}
\maketitle

\begin{abstract}
 The purpose of this paper is to establish that for any compact, connected $C^\infty$ Riemannian manifold there exists a robust family of kernels of increasing smoothness that are well suited for interpolation. They generate Lagrange functions that are uniformly bounded and decay away from their center at an exponential rate.
 An immediate corollary is that the corresponding Lebesgue constant will be
 uniformly bounded with a constant whose only dependence on the set of data sites is reflected
in the \emph{mesh ratio}, which measures the uniformity of the data.

The analysis needed for these results was inspired by some fundamental work of Matveev where the Sobolev decay of Lagrange
functions associated with certain kernels on $\Omega \subset \reals^d$ was obtained. With a bit more work, one establishes the following: Lebesgue constants
associated with surface splines and Sobolev splines are uniformly bounded on $\reals^d$ provided the data sites $\Xi$ are quasi-uniformly distributed.
The non-Euclidean case is more involved as the geometry of the underlying surface comes into play.
In addition to establishing bounded Lebesgue constants in this setting, a ``zeros lemma'' for compact Riemannian manifolds is established.
\end{abstract}

\section{Introduction}
Radial Basis Functions (RBFs) and their counterparts on the sphere (SBFs) are now well-known
and widely used for fitting a surface to scattered data arising from sampling an unknown function defined on $\reals^d$ ( or $\sphere$). The fitting of the surface is typically implemented by means of interpolation or discrete least squares. Both of these tools fit under the larger umbrella of kernel based approximation, a methodology, suitable for treating functions defined on very general domains. 
This approach has achieved success in approximating scattered data on spheres and 
Euclidean domains 
(including the case when data is arranged on lower dimensional manifolds),
as well as more esoteric domains like graphs and Lie groups \cite{Castel-etal-07-1,Erb-etal-08-1,Gutzmer-96-1,jones-etal-2008}.  
The hallmark of this methodology is to use the kernel $\kappa$ to create an approximant, $s$, by taking linear combinations of  the form $s = \sum_{j=1}^N A_j \kappa(\cdot, x_j)$. 
Evidently the first  challenge is to find a ``scheme'' assigning the coefficients and sometimes the ``centers'' $x_j$, $j=1...N$. 

One such a scheme, interpolation, is known to work quite well and the interpolants even exhibit ``local'' behavior just as in the case of univariate spline interpolation. 
The reasons for the local behavior of spline interpolation, i.e., where changing the data locally only significantly alters the interpolant locally, are well understood. 
In fact since univariate splines are compactly supported, the corresponding interpolation (or collocation) matrices are banded. Thus, whenever the interpolation matrices are invertible, the inverse matrices decay exponentially away from the main diagonal (see \cite{Demko} and \cite{DMS}). 
This fact, together with the local support of spline functions, implies the exponentially fast decay of Lagrange functions away from their ``centers''.

In contrast, many popular kernels (including RBFs and SBFs) are typically globally supported and the 
interpolation matrices,
$\calc_{\Xi}:=(\kappa(\xi,\zeta))_{(\xi,\zeta)\in \Xi}$,
are full. Even in the case of the compactly supported Wendland functions, as the number of data sites increase while still interpolating with shifts of a fixed $\phi$, the bands of the interpolation matrices become ever larger. 

Nevertheless, in this paper, we will show that for 
any compact, complete Riemannian manifold there exists a robust family of kernels of increasing smoothness, 
$\{\kappa_{m,\M}: m\in \ints, m>d/2\}$ (see Section~ \ref{native_space_kernels} for a definition),
generating Lagrange (or fundamental, or cardinal) functions --
the functions $\chi_{\xi} = \ \sum_{\zeta\in \Xi} A_{\zeta} \kappa(\cdot, \zeta)$ satisfying
$\chi_{\xi}(\zeta)= \delta(\zeta,\xi)$ --
that are uniformly bounded and decay away from their center at an exponential type rate. 
An immediate corollary  is that the corresponding Lebesgue 
constant will be uniformly bounded with a constant 
whose only dependence on $\Xi$ is reflected in the mesh ratio, 
which measures the uniformity of the data, a simple parameter measuring the ``badness'' of the geometry of the data  
(measured as the ratio of mesh norm to separation radius -- see definitions below). 

The Lebesgue constant, 
$L:=L(\Xi):=\sup_{x\in \M} \sum_{\xi\in \Xi}|\chi_{\xi}(x)|$, 
is  the $L_{\infty}\to L_{\infty}$ operator norm of the projector 
that maps continuous functions to interpolants, 
$f \mapsto I_{\Xi} f:=\sum_{\xi \in \Xi}f(\xi)\chi_{\xi}$.
This constant provides a measure of the stability of the interpolation process. 
This, it should be noted, is independent of the myriad ways the  interpolant may be determined; 
how coefficients are determined depends strongly on the basis for the space of interpolants -- 
using the kernels $\bigl(\kappa(\cdot,\xi)\bigr)_{\xi\in\Xi}$ as the basis for this space
may lead to  coefficients of indeterminant size, 
$(c_{\xi})_{\xi\in\Xi}:= \calc_{\Xi}^{-1} \bigl(f(\xi)\bigr)_{\xi\in \Xi}$,
while using the basis of Lagrange functions always involves a stable selection of coefficients. Thus, the boundedness of the Lebesgue constant indicates the stability of 
the kernel based interpolation we consider, while the rapid decay of the Lagrange functions shows that this kind of interpolation is local in nature.

The challenge of finding stable interpolation processes is underlined by the
fact that, for many of the basic tools in approximation theory (e.g., polynomials, trigonometric functions, spherical harmonics), interpolation is known to be \emph{unstable}.
A consequence of our main result is that for any $\sphere$, the family of SBFs $\kappa_{m,\sphere}$ give rise to bounded Lebesgue constants (each corresponding, roughly, to the fundamental solution of an elliptic differential operator). 
This result stands in stark contrast with any {\em polynomial} interpolation scheme on the sphere. In particular, let $\Pi_L$ be the space of all spherical harmonics of total degree $L$ or less and let $T$ be any bounded linear projector (i.e., satisfying $T^2=T$)
$$T: C(\sphere) \rightarrow \Pi_L$$
with both spaces endowed with the $L^{\infty}(\sphere)$ norm. Then, as shown in
 \cite{Reimer},
$$
||T||_{\infty} 
\geq 
L^{\frac{d-1}{2} }, \enskip d \geq 2.
$$
Thus, if $\Xi$ has mesh norm $h$ and $h \approx q$, 
then $\# \Xi = {\mathcal O}(h^{-n}) \approx \dim \Pi_{h^{-1}}$.
In this case any polynomial interpolation scheme would have Lebesgue constant
$\geq {\cal O}\bigl((\frac{1}{h})^{(d-1)/ 2}\bigr).$

Another happy consequence of the uniform boundedness of the Lebesgue constant is
the so-called Lebesgue lemma, which states that interpolation is {\em near best} $L_{\infty}$ approximation from the space of interpolants,
$S_{\Xi}= \spam \{ \kappa(\cdot,\xi): \xi \in \Xi\}$. 
Since $I_{\Xi} s = s$ for each $s\in S_{\Xi}$, one has
$$\|f-I_{\Xi}f\| \le  \|f-s -I_{\Xi}(f-s)\|_{\infty} \le (1+L) \d(f,S_{\Xi}).$$
This is significant for kernel based approximation,  because $L_{\infty}$ error estimates for 
interpolation have been notably imprecise. By using direct methods of error analysis, it has, to date, only been possible to measure error for target functions coming from a certain reproducing kernel Hilbert
space associated with $\kappa$. (Frequently this space is an $L_2$ Sobolev space; for
an example of this kind of error estimate, see Lemma \ref{NSEE} of this paper.)
Because one cannot assume that data is generated by a target function from a specific smoothness space, it is desirable to have an accurate view of approximation power for target functions of 
arbitrary smoothness, e.g., by having error estimates for a scale of spaces capturing fractional smoothness.
Unfortunately, precise error estimates for interpolation from the scale of H{\"o}lder spaces ($C^s$), Bessel potential spaces ($W_{\infty}^s$), Besov spaces ($B_{\infty,q}^s$) or other spaces measuring smoothness in $L_{\infty}$ have not been forthcoming.
Recently, this has been addressed for kernel based approximation, using schemes other than interpolation, \cite{Hang, MNPW}, and the results of this article point the way toward fixing a significant gap in the study of kernel interpolation.

To illustrate this, we compute the SBF, $\kappa_{2,\sph^2}$, on the sphere in $\reals^3$: $\sph^2$. This is a kernel having a bounded Lebesgue constant, but it
also belongs to the scheme of kernels considered in \cite{MNPW}, and, hence, its
best approximation properties are well known. It is associated with a 
fourth order elliptic partial differential operator on $\sph^2$, and provides
{\em approximation order} $\sigma$ for functions having smoothness $\sigma$ in $L_{\infty}$,
provided $\sigma \le 4$. This means that for a function, $f$, having smoothness $\sigma$,
$\d(f, S_{\Xi})=\mathcal{O}(h^{\sigma})$, where $h$ is the mesh norm of $\Xi$.
Previously, the most one could have said about approximation power
of interpolation with this kernel is that $\|I_{\Xi}f - f\|_{\infty} =  \mathcal{O}(h^{1})$ 
for functions, $f$, having two derivatives in $L_2$.
Thus, the Lebesgue lemma ensures that interpolants provide optimal $L_{\infty}$ approximation orders, and not only for continuous target functions of low smoothness; it also significantly improves the approximation power for functions having greater smoothness. 

\subsection{Methodology}
There is a unifying thread that explains the fast decay of the Lagrange functions for both the univariate splines and kernels we consider. 
Both interpolation problems can be considered as minimization problems with interpolatory constraints: $I_{\Xi}f = \mathrm{argmin}\{\|s\|:s_{|_{\Xi}} = f_{|_{\Xi}}\}$. This variational
aspect of the theory plays a well-known role in the etymology of the term spline -- 
the draftsman's spline is a flexible implement that passes through points fixed on drawingboard by assuming a smooth curve that minimizes a ``bending energy,'' the $L_2$ norm
of the second derivative.
The connection between kernel interpolation and
variational problems constrained by interpolatory conditions in Hilbert spaces 
is the basis of the approach taken by Madych and Nelson (cf. \cite{MaNe}), 
which has since become a fundamental part of RBF and SBF theory.  
However, the idea of cleverly exploiting this variational property to obtain decay of Lagrange functions was developed by Matveev in an impressive paper \cite{Mat}, that has gone mostly ignored (Johnson's use of Matveev's results to achieve surface spline $L_p$ saturation orders, \cite{MJ}, is a notable exception).

The subject of Matveev's interpolation problem is the $D^m$-spline, the interpolant minimizing a Sobolev seminorm $| f |_{m,\Omega} = \left( \sum_{|\alpha| =m} \int_{\Omega}  |D^{\alpha} f(x)|^2\dif x\right)^{1/2}$.  For general $\Omega$, the minimizing kernel is computationally intractable, but the case  $\Omega= \reals^d$ gives one of the pre-eminent families of RBFs, the well-known surface (or polyharmonic) splines considered by Duchon \cite{D1, D2} and many others.

Matveev's technique establishes a `Sobolev decay' of the Lagrange functions, i.e., the decay of the Sobolev seminorm of $\chi_{\xi}$ --- the Lagrange function for surface spline interpolation --- over the punctured plane,  $\reals^d \setminus B(\xi,R)$, having a hole of radius $R$ (cf. Section~\ref{lagrange_function}). Johnson applied a Sobolev embedding argument to obtain pointwise estimates
\cite[Corollary 2.3]{MJ} for Lagrange functions whose centers $\xi$ come from a fairly regular set of
data sites $\Xi$. It is clear that exponential decay such as this holding for {\em all} Lagrange functions 
would be sufficient to uniformly bound the `boundary-free' Lebesgue constant. A minor, albeit oblique, modification of his argument would show this. We remark that Johnson's paper had a completely different focus than stability of surface spline interpolation, and exponential bounds on the Lagrange functions were ancillary to its main goal.

In this article we have improved the approach to estimating Lagrange functions. Although it has been adapted to handle interpolation on compact Riemannian manifolds, the setting for the core of the argument treating Lagrange functions is the tangent space at $\xi$;
it is essentially still an argument set in $\reals^d$. 
The article is roughly divided into a component where the geometry of the manifold is important, where a powerful relationship between Sobolev norms on the manifold and on $\reals^d$ is established, 
and a second component, in which the Riemannian geometry is not so important,
that gives Sobolev and pointwise decay of the Lagrange functions 
modeled on Matveev's argument --- although with a different underlying variational problem, and with a different embedding argument, providing stronger pointwise estimates. 

\subsection{Outline}
Section~\ref{riemannian_manf} is devoted to Riemannian manifolds. Since this paper is about approximation theory,  not differential geometry \emph{per se}, we provide, in Section~\ref{background_manf}, the requisite geometric background on Riemannian manifolds -- discussing geodesics, Christoffel symbols, exponential maps, and so forth. Section~\ref{covariant_bounds} concerns the covariant derivative, which is essential for defining invariant Sobolev spaces. In particular, we provide various pointwise bounds on such derivatives in terms of ordinary (flat space) derivatives. These bounds are what allow us to use Euclidean-space Sobolev estimates to obtain similar bounds for the covariant Sobolev spaces on  a Riemannian manifold.

The covariant Sobolev spaces we will work with are due to Aubey (cf. \cite{Aub}); we define and discuss them in Section~\ref{sobolev_manf}.  In Section~\ref{metric_equiv}, we use the covariant-derivative bounds derived in Section~\ref{covariant_bounds} to provide a very strong equivalence between Sobolev norms over regions on the manifold and corresponding regions in the tangent space, with constants of equivalence that are independent of the geometry of the region. Section~\ref{sobolev_zeros} treats the problem of estimating functions having many zeros in a star-shaped domain in domain in $\reals^d$, and from that and the strong equivalence above, deduce similar results for a geodesic ball on a manifold. This is a key element to estimating Lagrange functions, which, of course, have many zeros. A consequence of this ``zeros lemma'' is a conventional uniform estimate on the interpolation error for many kinds of kernels. In Section~\ref{native_space_kernels}, we introduce the family of kernels associated with bounded Lebesgue constants. These are reproducing kernels for the invariant Sobolev spaces $W_2^m(\M)$ mentioned above.

In Section~\ref{lagrange_function}, we provide pointwise estimates of Lagrange functions and the main theorem, which estimates the size of the Lebesgue constant for interpolation on manifolds or on $\reals^d$ with surface splines or Sobolev splines. In Section~\ref{lagrange_function_bnds}, a bound on the size (in $L_{\infty}$) of Lagrange functions from a very general class of kernels is presented (kernels having a Sobolev space as their native space). 
Section ~\ref{lagrange_function_decay} treats the decay of the Lagrange functions from the class of kernels associated with the inner product of Section~\ref{sobolev_manf}.
This is done in roughly three stages. 
The first stage provides the basis for a ``bulk chasing'' argument by showing that the bulk of the tail of the Sobolev norm is contained in a thin annulus near the beginning of the tail. The second stage iterates this result, showing that the tail of the Sobolev norm decays exponentially. 
The third stage uses the zeros lemma
to provide the exponentially decaying pointwise estimates of the Lagrange function.
Section~\ref{lebesgue_const_bnd} gives the main result, which follows in a direct way from the decay
of the Lagrange functions.

The concrete, spherical, example of the SBF $\kappa_{2, \sph^2}$ is explicitly calculated in Section~\ref{example_S2}. It is shown that this SBF is a perturbation of the fundamental solution of the bi-Laplacian, $\Delta^2$ on $\sph^2$, and hence is of the type treated in \cite{MNPW}. 
In particular, it inverts a fourth order differential operator, and is, for this reason, not very different from either the thin plate splines in $\reals^2$ or the univariate cubic splines. Since much is known about the abstract approximation properties of these types of SBFs (but little is known about the convergence of interpolation), an application of Lebesgue's lemma reveals that interpolation provides $L_{\infty}$ approximation orders
commensurate to the smoothness of the function. In particular, 
 $\|I_{\Xi}f - f\|_{\infty} = \mathcal{O}(h^4)$ for functions in $C^4(\sph^2)$.

\subsection{Notation}
$\M$ denotes a compact, complete Riemannian manifold of dimension $d$.  More will be said about it in Section~\ref{riemannian_manf}, but for now we simply remark that it is a compact manifold -- a
Hausdorff space possessing a (maximal) family of charts $(\phi_{j},U_{j})_{j\in \mathcal{J}},$ where each map $\phi_{j}:U_{j}\subset\reals^d\to V_j\subset\M$ is a bijection and the compatibility condition $ \phi_{k}^{-1}\circ\phi_{j}\in C^{\infty}\bigl(\phi_j^{-1}(U_k)\bigr),$ for every $k,j\in \mathcal{J}$ is satisfied.
At every point $p\in \M$, we denote the tangent space by $T_pM$. This fact gives rise
to a notion of arc length of curves, a metric: $(\xi,\alpha)\mapsto \d(\xi,\alpha)$, and
measure, which we denote by $\mu$.

The basic tools used for estimating Lagrange
functions will be the various Sobolev norms and seminorms on balls, complements of balls and annuli.

Although the formulation for Sobolev spaces on manifolds is rather technical, hence postponed until the next section, the Euclidean versions are standard and quite simple.
For an open set $\Omega\subset \reals^d$, and a positive integer $j$, we define the seminorm of the function $u:\Omega \to \reals$ as 
$$|u|_{j,\Omega}:=|u|_{W_2^j(\Omega)}:=\left(\sum_{|\beta| = j}\binom{j}{\beta}  \int_{\Omega} \left|D^{\beta}u (x)\right|^2 \dif x \right)^{1/2}.$$
where the multinomial coefficient is $\binom{j}{\beta} := \frac{j!}{\beta_1!\dots \beta_d!}.$ The Sobolev norm is then
$$\|u\|_{j,\Omega}:=\|u\|_{W_2^j(\Omega)} := 
\left( \sum_{j=0}^m |u|_{j,\Omega}^2\right)^{1/2}.$$

Frequently we must simultaneously refer to sets on the manifold and on Euclidean space. To avoid confusion, we adopt the following notation.
We denote the ball in $\reals^d$ centered at $x$ having radius $r$ by $B(x,r).$ 
For $t>0$, we define the annulus  in $\reals^d$ centered at $x$ of thickness $r$ and outer radius $tr$  by
\begin{equation*}\label{Euc_annulus}
A(x, t,r):= B(x,tr) \backslash   B(x,(t-1)r).
\end{equation*}
We denote the ball in $\man$ centered at $\alpha$ having radius $r$ by $\b(\alpha,r).$ 
For $t>0$, we define the annulus in $\man$ centered at $\alpha$ by
\begin{equation*}\label{Man_annulus}
\a(\alpha,t, r):= \b(\alpha,tr) \backslash   \b(\alpha,(t-1)r).
\end{equation*}
The set $\Xi\subset\M$ will be assumed finite. Its  \emph{mesh norm}  (or \emph{fill distance}) $h:=\sup_{x\in \M} \d(x,\Xi)$ measures the density of $\Xi$ in $\M$, while the \emph{separation radius}
$q:=\frac12 \inf_{\substack{\xi,\zeta\in \Xi\\ \xi\ne \zeta}} \d(\xi,\zeta)$ determines the
spacing of $\Xi$. The \emph{mesh ratio} $\rho:=h/q$ measures the uniformity of the distribution of $\Xi$ in $\M$.
%

\section{Riemannian Manifolds}\label{riemannian_manf}
Since this paper is primarily about approximation,  we will provide the necessary background in Riemannian geometry -- e.g., Christoffel symbols, covariant and contravariant
tensors, exponential maps, etc. --  that will be needed here. For more information, see the books by do Carmo \cite{doC2,doC1}.

\subsection{Background}\label{background_manf}
Throughout our discussion, we will assume that $(\M,g)$ is a $d$-dimensional, connected, $C^\infty$  Riemannian manifold without boundary;  the Riemannian metric for $\M$ is $g$, which  defines an inner product $g_p(\cdot,\cdot )=\langle\cdot,\cdot \rangle_{g,p}$ on each tangent space $T_pM$. As usual, a chart is a pair $(\fraku, \phi)$ such that $\fraku\subset \M$ is open and the map $\phi:\fraku \to \reals^d$ is a one-to-one homeomorphism. An atlas is a collection of charts $\{(\fraku_\alpha,\phi_\alpha)\}$ indexed by $\alpha$ such that that $\M=\cup_\alpha \fraku_\alpha$ and, when $\phi_\alpha(\fraku_\alpha)\cap\phi_\beta(\fraku_\beta)\ne \emptyset$, $\phi_\beta \circ\phi_\alpha^{-1}$ is $C^\infty$. In a fixed chart  $(\fraku, \phi)$, the points $p\in \fraku$ are parametrized by $p=\phi^{-1}(x)$, where $x=(x^1,\ldots,x^d)\in U=\phi(\fraku)$.  As is common in differential geometry, we will use superscripts to denote coordinates. Thus, $x^j$ is the $j^{\text{th}}$ local coordinate for $p$ relative to the chart. Most of the analysis we will do takes place in local coordinates.

In such coordinates for $\fraku$, the tangent space $T_pM$ at $p$ has a basis comprising the partial derivatives $\left(\frac{\partial}{\partial x^j}\right)_p$, where $j=1,\ldots,d$. This allows us to define smoothly varying vector fields on $U$ via
\[
\bfe_j(x):= \frac{\partial}{\partial x^j},\ j=1,\ldots,d,\ x\in U=\phi(\fraku).
\]
For $f\in C^\infty(\fraku)$, so that $f\circ \phi^{-1}\in C^\infty(U)$, the action of a vector field $\bfv$ thus is given by the following:
\[
\bfv(f)(p) = \sum_{j=1}^d v^j(x) \frac{\partial (f\circ \phi^{-1})}{\partial x^j}= \sum_{j=1}^d v^j \bfe_j(f),
\]
for any $x=\phi(p)\in U$. The coefficients  $ v^j=v^j(x)$ are smooth functions; they are called the 
\emph{contravariant components} of $\bfv$ relative to the $\bfe_j$ basis.  We write $\bfv(f)(p)$, rather than  $\bfv(f)(x)$ because the expression on the left is independent of the choice of coordinates for $p$. Obviously, in terms of tangent vectors alone we have that
\[
\bfv = \sum_{j=1}^d v^j \bfe_j.
\]

\paragraph{Metric tensor} Let $v^j$ and $w^j$ be the contravariant components for the vectors $\bfv$ and $\bfw$ in $T_pM$. The inner product for the Riemannian metric $g$ is given is then given by
\[
\langle\bfv,\bfw \rangle_{g,p} = \sum_{i,j=1}^d g_{ij}(x)v^i(x) w^j(x), \ x=\phi(p) \ (\text{contravariant form)}.
\] 
The $g_{ij}$ are the \emph{covariant components} of $g$, and are given by  $g_{ij}(x) = \langle \bfe_i, \bfe_j \rangle_{g,p}$, which are the entries of the $d\times d$ Gram matrix for the basis $\{\bfe_i\}$. As was the case with $\bfv(f)(p)$,  $\langle\bfv,\bfw \rangle_{g,p}$ is itself independent of coordinates. In addition, if  $\bfv$ and $\bfw$ are $C^\infty$ vectors fields in $p$, then $\langle\bfv,\bfw \rangle_{g,p}$ is also $C^\infty$. Finally, because $g_{ij}$ is a Gram matrix, it is symmetric and positive definite.

The space dual to $T_p M$ is the cotangent space, and it is denoted by $T_p^\ast M$. 
Corresponding to the basis $\bfe_j=\frac{\partial}{\partial x^j}$ for tangent vectors, 
we have the dual vectors $\bfe^j = dx^j$.
The two bases are related through the metric via
\[
\bfe^j(x) = \sum_{k\in \D}g^{jk}(x)\bfe_k(x) \ \text{and}\ \bfe_k(x) = \sum_{j\in\D}g_{kj}(x)\bfe^j(x).
\]
The matrix with entries $g^{jk}$ is simply the inverse of the metric tensor $(g_{jk})$.
Making the usual identification between a vector space equipped with an inner product and its dual space, we may represent a vector $\bfv$ in terms of the $\bfe^j$'s as $\bfv =\sum_{j\in \D}v_j\bfe^j$. In this case the $v_j$'s are called the \emph{covariant components} of $\bfv$. The natural inner product of vectors $\bfv=\sum_{j\in \D}  v_j\bfe^j$ and $\bfw=\sum_{j\in \D}  w_j\bfe^j$ expressed in covariant components is 
\[
\langle \bfv, \bfw\rangle_{g,p} 
=
\sum_{i,j=1}^d g^{ij}(x) v_i(x) w_j(x) \  (\text{covariant form)}.
\]

An order $k$ \emph{covariant tensor} $\mathbf T$ is a real-valued, multilinear function of the $k$-fold tensor product of $T_pM$. We denote by $T_p^kM$ the covariant
tensors of of order $k$ at $p$. In terms of the local coordinates, there is a smoothly 
varying basis $\bfe^{i_1}\otimes \dots\otimes \bfe^{i_k}$ for the $k$-fold tensor product of tangent spaces. Thus, the covariant tensor field $\mathbf T$ of order $k$ on $\fraku$ can be written as
\[
\mathbf T = \sum_{\hati \in \D^k} T_{\hati}\, \bfe^{i_1}\otimes \cdots\otimes \bfe^{i_k},
\]
where we adopt the convention $\hati = (i_1,\dots,i_k)$. The $T_{\hati}$ are the covariant components of $\mathbf T$. The metric $g$ is itself an order 2 covariant tensor field. One can also define contravariant tensors and tensors of mixed type. 

Because $T_p^kM = T_pM\otimes \cdots\otimes T_pM$ (k times), the metric $g$ induces a natural, useful, invariant inner product on $T_p^kM$; in terms of covariant components, it is given by
\begin{equation}
\label{metric_inner_prod_tensor}
\langle \mathbf S, \mathbf T\rangle_{g,p} = \sum_{\hati,\hatj \in \D^k} g^{i_1 j_1}\cdots g^{i_k j_k}S_\hati \,T_\hatj\,.
\end{equation}
The corresponding norm will be denoted by $|\mathbf T|_{g,p}$. We will need to compare this inner product, with its metric $g$, to one using the Euclidean metric, which we denote  by  $\delta$, 
where $\delta^{i,j}$ is the Kronecker $\delta$.
Thus, for the Euclidean case, which depends on the chart $(\fraku,\phi)$, we have
\begin{equation}
\label{euclid_inner_prod_tensor}
\langle \mathbf T, \mathbf S\rangle_{\delta,x} 
= 
\sum_{\hati,\hatj \in \D^k}
\delta^{i_1,j_1}\cdots \delta^{i_k,j_k} 
T_{\hati}S_{\hatj} = 
\sum_{\hati  \in \D^k}
   T_{\hati}S_{\hati}\,,
\end{equation}
with $|\mathbf T|_{\delta,x}$ being the corresponding norm. In the sequel, we will need the following lemma comparing lengths of tensors in these two inner products.

\begin{lemma} \label{inner_prod_tensors}
Let $x\in U=\phi(\fraku)$. If $ \Lambda_x(g)$ and $ \lambda_x(g)$ are the largest and smallest eigenvalues of $g^{i,j}(x)$, then
\begin{equation}\label{bounds_tensor_ip}
\lambda_x(g)^k \langle \mathbf T,\mathbf T\rangle_{\delta,x}\le \langle \mathbf T,\mathbf T \rangle_{g,p}  \le \Lambda_x(g)^k \langle  \mathbf T,\mathbf T\rangle_{\delta,x}.
\end{equation}
\end{lemma}
\begin{proof}
The matrix $g^{i,j}$ is positive definite, and so each eigenvalue $\lambda_m$ is positive. If necessary, we repeat an eigenvalue. Let $\bfu_m$ be the corresponding eigenvector. The set of eigenvectors is assumed to be orthonormal. If  $u_m^i$ is the $i^{th}$ entry in the column $\bfu_m$, the spectral theorem implies that
\[
g^{i,j} = \sum_{m\in \D} \lambda_m u_m^i u_m^j \quad \mbox{and} \quad  \delta^{i,j} = \sum_{m\in \D} u_m^i u_m^j .
\]
Using this in connection with the expression for $\langle  \mathbf T,\mathbf S \rangle_{g,p}$ yields
\begin{eqnarray*}
\langle  \mathbf T,\mathbf T\rangle_{g,p} 
&=& 
\sum_{\substack{(m_1\dots m_k) \\ \in \D^k} }
  \lambda_{m_1} \cdots \lambda_{m_k}  
  \left(
    \sum_{\substack{(i_1\dots i_k) \\ \in \D^k}}
      u_{m_1}^{i_1}\cdots u_{m_k}^{i_k} 
      T_{i_1,\ldots,i_k}
    \right)^2\\
&\ge& 
\lambda_x(g)^k 
\left(
    \sum_{\substack{(i_1\dots i_k) \\ \in \D^k}}
      \sum_{\substack{(j_1\dots j_k) \\ \in \D^k}}
        \sum_{\substack{(m_1\dots m_k) \\ \in \D^k} }
           u_{m_1}^{i_1}u_{m_1}^{j_1}\cdots u_{m_k}^{i_k}u_{m_k}^{j_k}
           T_{i_1,\ldots,i_k}T_{j_1,\ldots,j_k}
\right)\\
&= &
\lambda_x(g)^k 
   \sum_{\substack{(i_1\dots i_k) \\ \in \D^k}}
     \left(T_{i_1,\ldots,i_k}\right)^2 
=
\lambda_x(g)^k \langle  \mathbf T,\mathbf T\rangle_{\delta,x},
\end{eqnarray*}
which establishes the lower bound in (\ref{bounds_tensor_ip}). The upper bound follows in the same way.
\end{proof}

\paragraph{Geodesics} A curve $\gamma(t)$, defined by $(x^j(t))$ in local coordinates, has its tangent vector given by 
$ \dot \gamma =\sum_{j\in \D}\frac{dx^j}{dt}\bfe_j$. 
The arc length of  $\gamma(t)$ is defined by
\[
s_\gamma(\tau) = \int_0^\tau \sqrt{\sum_{i,j=1}^d  g_{ij} \frac{dx^i}{dt} \frac{dx^j}{dt}}d t .
\]
If $p,q$ are points in $\M$, then the \emph{distance} between $p$ and $q$, $\d(p,q)$, is defined to be the infimum of the length of all piecewise differentiable curves joining $p$ and $q$ \cite[\S7.2]{doC1}.

Extremals of the arc length functional are called \emph{geodesics}. If we use the arc length 
$s$ as the parameter (i.e., $t\to s$), then the Euler-Lagrange equations for 
$x^k(s)$ are the following second order, nonlinear differential equations:
\begin{equation}
\label{geodesic_difeq}
\frac{d^2 x^k}{ds^2} +  
\sum_{i,j=1}^d\Gamma^k_{ij} \frac{dx^i}{ds}\frac{dx^j}{ds} = 0.
\end{equation}
\label{christoffell_def}
The $\Gamma^k_{ij}$ are called \emph{Christoffel symbols}; they are given by
\begin{equation}
\Gamma^k_{ij} = \frac12 \sum_{m\in \D}g^{km} \left( \frac{\partial g_{jm}}{\partial x^i} + \frac{\partial g_{im}}{\partial x^j} - \frac{\partial g_{ij}}{\partial x^m} \right).
\end{equation}

A geodesic solving (\ref{geodesic_difeq}) is specified by giving an initial point 
$p\in \M$, whose coordinates we may take to be $\bigl(x^1 (0),\dots,x^d (0)\bigr)=0$, 
together with a tangent vector  $\bft_p$ having components $\frac{dx^i}{ds}(0)$. A Riemannian manifold is said to be \emph{complete} if the geodesics are defined for all values of the parameter $s$. All compact Riemannian manifolds are complete. $\reals^d$ is complete.

\paragraph{Exponential map} We define the \emph{exponential map} $\Exp_p: T_pM \to \M$ by letting 
$\Exp_p(0)=p$ and $\Exp_p(s\bft_p)=\gamma_p(s)$, 
where $\gamma_p(s)$ is the unique geodesic that passes through $p$ for $s=0$ 
and has a tangent vector $\dot \gamma_p(0)=\bft_p$ of length 1; i.e.,  
$\langle \bft_p,\bft_p\rangle_{gp}.=1$.  By the Hopf-Rinow Theorem \cite[\S 7.2]{doC1} $\M$ is complete if and only if $\Exp_p$ is defined on all of $T_pM$. Other important consequences of this theorem are that $\M$ is a complete metric space under the distance $\d(p,q)$, that any two points $p,q\in \M$ are connected via a geodesic of minimum length $\d(p,q)$, and that the exponential map is defined on the whole tangent bundle $TM$; as such, it is a smooth function of both arguments \cite[\S 9.3]{hicks1965}.

Although geodesics having different initial, non-parallel unit tangent vectors $\bft_p=\dot \gamma_p(0)$ may eventually intersect, there will always be a neighborhood $\fraku_p$ of $p$ where they do not.  In $\fraku_p$, the initial direction $\bft_p$,  together with the arc length $s$, uniquely specify a point $q $ via $q=\gamma_p(s)$, and the exponential map $\Exp_p$ is a diffeomorphism between the corresponding neighborhoods of $0$ in  $T_pM$ and $p$ in $M$. In particular, there will be a largest ball $B(0,\mathrm{r}_p)\in T_pM$ about the origin in $T_pM$ such that $\Exp_p: B(0,\mathrm{r}_p)\to \b(p,\mathrm{r}_p)\subset \M$ is injective and thus a diffeomorphism; $\mathrm{r}_p$ is called \emph{injectivity radius} for $p$. By choosing cartesian coordinates on $B(0,\mathrm{r}_p)$, with origin $0$, and using the exponential map, we can parametrize $\M$ in a neighborhood of $p$ via $q=\Exp_p(x)$, $x\in T_pM$. Finally, any chart with $\fraku=\exp_p(B(0,\mathrm r)$, $\mathrm{r} 
 \le \mathrm{r}_p$ and $\phi=\Exp_p^{-1}$ is called a \emph{normal chart} of radius $\mathrm{r}$  and corresponding cartesian coordinates are called a \emph{normal} coordinates. 

The \emph{injectivity radius} of $\M$ is  $\inj:=\inf_{p\in M}\mathrm{r}_p$. If $0<\inj\le \infty$, the manifold is said to have \emph{positive} injectivity radius. All compact Riemannian manifolds without boundary are both complete and have positive injectivity radius. On the other hand, many noncompact Riemannian manifolds, including the simple case of $\reals^d$, also have positive injectivity radii. In fact, the injectivity radius of $\reals^d$ is $\infty$. It is worth pointing out that manifolds having $\inj>0$ are always complete \cite[\S 1.1]{eichhorn2007}

We make special note of the fact that, for a compact Riemannian manifold, the family of exponential maps are uniformly isomorphic; i.e., there are constants $0<\Gamma_1 \le \Gamma_2<\infty$ so that
for every $p_0\in \M$ and every $x,y\in B(0,\mathrm r)$
\begin{equation}\label{isometry}
 \Gamma_1|x-y|\le \d(\Exp_{p_0}(x), \Exp_{p_0}(y)) \le \Gamma_2 |x-y|.
\end{equation}
We will have more to say about this later in Remark~\ref{det_dist_normal}.

\subsection{Covariant derivative bounds}\label{covariant_bounds}
We now turn to the concept of higher order covariant differentiation on manifolds. Such derivatives are intrinsic; they are used to define Sobolev spaces on measurable subsets of $\M$. After relating  covariant derivatives and ordinary partial derivatives in equations (\ref{kth_order_covariant_deriv_tensor}) and (\ref{kth_order_derivatives}), we obtain bounds that will be used  to prove Lemma \ref{Fran}, which is a key result 
concerning local norm equivalence between Sobolev spaces on the manifold and 
Sobolev spaces on normal charts.

The \emph{covariant derivative}, or \emph{connection}, $\nabla$ associated with $(\M,g)$ is defined as follows. Let 
$\bfv=\sum_{j=1}^d v^j\bfe_j$ and 
$\bfw=\sum_{j=1}^d w^j\bfe_j$ be vector fields. 
The covariant derivative of $\bfw$ in the \emph{direction} of the vector field $\bfv$ is given in terms of the Christoffel symbols,
\[
\nabla_\bfv \bfw := 
\sum_{k \in \D} 
  \sum_{i\in \D}
    \bigl[
    (v^i\frac{\partial w^k}{\partial x^i} + 
    \sum_{j\in \D}
      \Gamma^k_{ij}v^iw^j)\bigr]\bfe_k\,,
\]
and is itself a vector field. If we use a covariant vector field, $\bfw = \sum_{k\in \D}w_k\bfe^k$, then 
\[
\nabla_\bfv \bfw  
= 
\sum_{k \in \D} 
  \sum_{i\in \D}
    \bigl[(v^i\frac{\partial w_k}{\partial x^i} - 
      \sum_{r\in \D} \Gamma^r_{ik}v^i w_r)
    \bigr]\bfe^k\,.
\]
This is a directional derivative. The appropriate ``gradient'' is a rank 2 tensor,
\[
\nabla \bfw 
= 
  \sum_{(k,i)\in \D^2}
    \left(
      \frac{\partial w_k}{\partial x^i} - 
      \sum_{s\in \D}
        \Gamma^s_{ik}w_s
    \right)
    \bfe^k\otimes\bfe^i
\]
Covariant derivatives may be defined for any type of tensor field -- covariant, contravariant, or mixed. For instance, if $\mathbf T$ is the order $k$ (covariant) tensor defined by
\[
\mathbf T
=
\sum_{\hati  \in  \D^k}T_{\hati} \bfe^{i_1}\otimes \cdots \otimes \bfe^{i_k},
\]
then the covariant derivative of $\mathbf T$ is 
\[
\nabla  \mathbf T 
= 
\sum_{j\in\D}
\sum_{\hati \in  \D^k}
\left(
  \frac{\partial T_{\hati}}{\partial x^j} 
  - 
  \sum _{r=1}^k \sum_{s\in \D}
    \Gamma^s_{j, i_r} T_{i_1,\ldots,i_{r-1},s, i_{r+1},\ldots, i_k}
\right) 
\bfe^{i_1}\otimes \cdots \otimes \bfe^{i_k}\otimes \bfe^j,
\]
which is an order $k+1$ covariant tensor. Of course, if $T = f$ is a scalar valued function, then $\nabla f = \sum_{j\in \D} \frac{\partial f}{\partial x^j}\bfe^j$, (or, more precisely, $\nabla f = \sum_{j\in \D} \frac{\partial f\circ\phi^{-1}}{\partial x^j}\bfe^j$) which is just the usual expression for the gradient. 
The ``Hessian'' is $\nabla^2 f $; it has the form
\[
\nabla^2 f 
=
\sum_{(i,j)\in\D^2}
\left(
  \frac{\partial^2 f}{\partial x^i \partial x^j} 
  - 
  \sum_{k\in \D}
  \Gamma^k_{i,j}\frac{\partial f}{\partial x^k}
\right)
\bfe^i\otimes \bfe^j.
\]
The components of the $k^{th}$ covariant derivative of $f$ have the form
\begin{equation}
\label{kth_order_covariant_deriv_tensor}
(\nabla^k f(x))_{\hati} 
= 
(\partial^k f(x))_{\hati} 
+ 
\sum_{m=1}^{k-1}
  \sum_{\hatj \in \D^m}
    A_{\hati}^{\hatj}(x) (\partial^m f(x))_{\hatj} 
\end{equation}
where 
\[
(\partial^m f)_{\hatj} := \frac{\partial^m }{\partial x^{j_1}\cdots\partial  x^{j_m} } f\circ \phi^{-1} ,
\]
and where the coefficients $x\mapsto A_{\hati}^{\hatj}(x)$ depend on the Christoffel symbols and their derivatives to order $k-1$, and, hence, are smooth in $x$.  This can also be written in standard multi-index notation. Let $\alpha_1, \alpha_2,\ldots, \alpha_d$ be the number of repetitions of $1,2,\ldots, d$ in $\hatj $, and let $\alpha = (\alpha_1,\ldots,\alpha_d)$. Then,
\begin{equation}
\label{multi_index_superscript_notation}
(\partial^m f)_{\hatj} :=\frac{\partial^m }{\partial (x^{1})^{\alpha_1}\cdots\partial  (x^{d})^{\alpha_d} } f\circ \phi^{-1} =:  D^{|\alpha |}_\alpha f\circ \phi^{-1} , \  |\alpha | =\sum_{k=1}^d \alpha_k =m.
\end{equation}

Unlike ordinary partial derivatives $(\partial^k f(x))_{\hati}$ or even the Christoffel symbols, the components $(\nabla^k f(x))_{\hati}$ (and, of course, those of $\nabla \mathbf T$) transform \emph{tensorially} under a change of coordinates, so that $\nabla^k f$ is independent of the coordinate system. This was a discovery of Levi-Civita \cite[\S 2.3]{doC1}; it is why the covariant derivative $\nabla$ is an important improvement on the partial derivative  $\partial$.

Even though the partials are not invariant, we can express them in terms of covariant derivative components, provided we use local coordinates. Because of the relationship between the highest order derivative in (\ref{kth_order_covariant_deriv_tensor}) and the corresponding covariant derivative's component, we can solve for the highest derivative in terms of covariant derivatives:
\begin{equation}
\label{kth_order_derivatives}
 (\partial^k f(x))_{\hati}  = 
(\nabla^k f(x))_{\hati} + 
\sum_{m=1}^{k-1}
  \sum_{\hatj \in \D^m}
    B_{\hati}^{\hatj} (x)(\nabla^m f(x))_{j_1,\ldots,j_m}.
\end{equation}

The expressions in (\ref{kth_order_covariant_deriv_tensor})  and (\ref{kth_order_derivatives}) are needed to obtain bounds on $\langle \nabla^kf,\nabla^kf \rangle_{g,p}$. These bounds will prove useful in estimating certain invariant Sobolev norms using local coordinates. They are given below.

\begin{proposition} 
\label{pointwise_ineq} 
Let $(\fraku,\phi)$ be a chart, with $U=\phi(\fraku)$. For every $x=\phi(p)\in U$, there exist multivariate polynomials $\tilde c_U(g,k)\ge 1$ and $\tilde C_U(g,k)\ge 1$ in $g_{ij}(x)$, $g^{ij}(x)$ and their derivatives up to order $k$ such that
\begin{equation}
\label{euclidean_bnds_U}
 | \nabla^k f|_{\delta,x}^2 \le \tilde C_U(g,k)\sum_{m=1}^k | \partial^m f |_{\delta,x}^2 \ \text{ and }\  
  | \partial^k f |_{\delta,x}^2 \le \tilde c_U(g,k)\sum_{m=1}^k | \nabla^m f |_{\delta, x}^2 
\end{equation}
where $|\cdot |_{\delta,x}$  is the norm associated with the inner product in (\ref{euclid_inner_prod_tensor}).  In addition, we have that
\begin{equation}
\label{covariant_bnds_U}
c_U(g,k)\sum_{m=1}^k | \partial^m f |_{\delta,x}^2\le \sum_{m=1}^k | \nabla^m f |_{g,p}^2 \le C_U(g,k)\sum_{m=1}^k  | \partial^m f |_{\delta,x}^2,
\end{equation}
where $c_U(g,k)=\big(\sum_{m=1}^k\min(1,\lambda_x(g))^{-m} \tilde c_U(g,m))\big)^{-1}$ and $C_U(g,k)=\sum_{m=1}^k \Lambda_x(g)^m \tilde C_U(g,m))$; $\lambda_x(g)$ and $\Lambda_x(g)$ are the smallest and largest eigenvalues of $g^{ij}(x)$. 
\end{proposition}

\begin{proof}
From (\ref{kth_order_covariant_deriv_tensor}) and Schwarz's inequality, we see that
\begin{equation*}
|(\nabla^k f)_{\hati}|^2 \le A_{\hati}^2\left(|(\partial^k f)_{i_1,\ldots,i_k}|^2 +\sum_{m=1}^{k-1} | \partial^m f |_{\delta,x}^2\right) 
\end{equation*}
where $A_{\hati}^2 := 1+ \sum_{m=1}^{k-1}\sum_{\hatj\in\D^m}(A_{\hati}^{\hatj})^2$. 
Summing over the $i$ indices yields
\begin{equation}
 | \nabla^k f|_{\delta,x}^2 \le C_U(g,k) \sum_{m=1}^k | \partial^m f |_{\delta,x}^2,\ \tilde C_(g,k): = \sum_{\hati \in \D^k} A_{\hati}^2.
\end{equation}
Examining the coefficients in (\ref{kth_order_covariant_deriv_tensor}) shows that $C_U(g,k)$ is a polynomial in the Christoffel symbols and their derivatives. But these are linear combinations of derivatives and multiples of the metric components, so $\tilde C_U(g,k)$ is a polynomial in $g^{ij}$, $g_{ij}$ and their $x$ derivatives. Repeating the procedure with (\ref{kth_order_derivatives}) gives the result for $ | \partial^k f |_{\delta,x}^2$. The inequalities in (\ref{covariant_bnds}) follow from (\ref{euclidean_bnds}) and Lemma~\ref{inner_prod_tensors}.
\end{proof}

Let $q\in \M$. Suppose $0<\mathrm r<\inj$ and sufficiently small that, in a fixed chart $(\fraku,\phi)$, $\fraku$ contains a neighborhood of $q$ and all of the associated geodesic balls $\b(q,\mathrm r)$. As we mentioned earlier, the exponential map is a smooth on the tangent bundle. Consequently, if we let $x,y=\phi(q),\in U$,  then $p(x,y)=\Exp_{\phi^{-1}(y)}(x)\in \M$ is smooth in both $x$ and $y$. If $q$ is allowed to vary, then in a normal chart $(\b(q,\mathrm r), \Exp_q^{-1})$ the entries of the metric tensor $g_{i,j}$, its maximum and minimum eigenvalues, its determinant, and the Christoffel symbols vary smoothly in both $y=\phi(q)$ as well as $x$. They are thus uniformly continuous on the closure of $\cup_q \b(q,r)$, and are bounded there independently of $q$ and $x$. Since we can cover 
the compact manifold $\M$ with a finite number of such neighborhoods, the boundedness of the quantities associated with the metric tensor hold \emph{uniformly} for any normal coordinate system of radius $\mathrm r$. 

The manifold $\M$ doesn't have to be compact for the metric and related quantities, when expressed in normal coordinates on $\b(q,\mathrm r)$, to be bounded uniformly in $q$ by constants depending only on $\mathrm r$. This also holds for noncompact $C^\infty$ Riemannian manifolds having \emph{bounded geometry}, which means that they have positive injectivity radius and that for all $k\ge 0$, the curvature tensor $R$ satisfies $\sup_{p\in \M} |\nabla^kR|_{g,p}\le C_k$, where $C_k$ depends only on $k$ \cite{cheeger_etal1982, eichhorn1991, eichhorn2007, triebel1992}. Compact manifolds fall into this class, as does the space $\reals^d$.  The implication that is important here is the following result.

\begin{corollary}\label{chart_dependence}
Let $\M$ be a compact manifold, or, more generally, have bounded geometry. In addition, suppose $0<\mathrm r <\inj$ and $\fraku=\b(q,\mathrm r)$ and $\phi=\Exp_q^{-1}$. Then there are positive constants $\tilde c(\mathrm r,k)$, $\tilde C(\mathrm r,k)$, $c(\mathrm r,k)$, and $C(\mathrm r,k)$ such the following bounds hold on $\Exp_q^{-1}(\b(q,\mathrm r))$ independently of $q$:
\begin{equation}
\label{euclidean_bnds}
 | \nabla^k f|_{\delta,x}^2 \le \tilde C(\mathrm r,k)\sum_{m=1}^k | \partial^m f |_{\delta,x}^2 \ \text{ and }\  
  | \partial^k f |_{\delta,x}^2 \le \tilde c(\mathrm r,k)\sum_{m=1}^k | \nabla^m f |_{\delta, x}^2,
\end{equation}
where $|\cdot |_{\delta,x}$  is the norm associated with the inner product in (\ref{euclid_inner_prod_tensor}), and 
\begin{equation}
\label{covariant_bnds}
c(\mathrm r,k)\sum_{m=1}^k | \partial^m f |_{\delta,x}^2\le \sum_{m=1}^k | \nabla^m f |_{g,p}^2 \le C(\mathrm r,k)\sum_{m=1}^k  | \partial^m f |_{\delta,x}^2.
\end{equation}
\end{corollary}

\begin{proof}
All of the constants in $ \tilde c_U(g,m))$ and  $\tilde C_U(g,k)$ in (\ref{euclidean_bnds_U}) are polynomials in various quantities associated with $g_{ij}$. Since these quantities are all uniformly bounded by functions of $\mathrm r$, we see that there are new constants $\tilde c(\mathrm r,k)$ and $\tilde C(\mathrm r,k)$ such that $\tilde c_U(g,k) \le \tilde c(\mathrm r,k)$ and $\tilde C_U(g,k) \le \tilde C(\mathrm r,k)$. The bounds in (\ref{euclidean_bnds}) follow similarly.
\end{proof}

\begin{remark}\label{det_dist_normal}
{\rm From what we said earlier, we note that in a normal chart of radius $\mathrm r<\inj$, the measure on $\M$ has the form $d\mu = \sqrt{\det(g_{ij})} dx^1\cdots dx^d$. If $\M$ is compact or has bounded geometry, then there are positive constants $c_1(r),c_2(r)$ such that $c_1(\mathrm r)\le \sqrt{\det(g_{ij})}\le c_2(\mathrm r)$ on $B(0,\mathrm r)$.  We also point out that, for similar reasons, $\Gamma_1$ and $\Gamma_2$ in (\ref{isometry}) depend only on $\mathrm r$, and not the center $p_0$.}
\end{remark}

\section{Sobolev spaces on subsets of $\M$}\label{sobolev_manf}
Sobolev spaces on subsets of a Riemannian manifold can be defined in an invariant way, using covariant derivatives \cite{Aub}. In defining them, we will need to make use of the spaces $L_\sfp$, $L_\sfq$. To avoid problems with notation, we will use the sans-serif letters $\sfp$, $\sfq$, rather than $p$, $q$, as subscripts. Here is the definition.
\begin{definition}\label{Sob_Norm}
Let $\Omega\subset \man$ be a measurable subset. For all $1\le \sfp \le \infty$, we define the Sobolev space $W_\sfp^m(\Omega)$ to be all $f:\M \to \reals$ such that $ | \nabla f |_{g,p} $ in $L_\sfp$. The associated norms are as follows:
%
%
\begin{equation}\label{def_spn} 
\|f\|_{W_\sfp^m(\Omega)}:= 
\left\{
\begin{array}{cl}
\left( 
  \sum_{k=0}^m
  \int_{\Omega} 
    | \nabla^k f |_{g,p}^\sfp
  \, \dif \mu(p)\right)^{1/\sfp}, & 1\le \sfp <\infty; \\ [5pt]
  \max_{\,0\le k\le m} \bigl\|  | \nabla^k f |_{g,p} \bigr\|_{L_\infty(\Omega)}, & \sfp=\infty.
\end{array}
  \right.
\end{equation}
When $\sfp=2$, the norm comes from the Sobolev inner product
\begin{equation}\label{def_sn} 
\langle f, g\rangle_{m,\Omega}:=\langle f,g\rangle_{W_2^m(\Omega)}:= 
  \sum_{k=0}^m
  \int_{\Omega} 
    \left\langle
      \nabla^k f,  \nabla^k g
    \right\rangle_{g,p}
  \, \dif \mu(p).
\end{equation}
We also write the $\sfp=2$ Sobolev norm as $\|f\|_{m,\Omega}^2 := \langle f,f\rangle_{m,\Omega}$.
When $\Omega=\man$, we may suppress the domain: $\langle f,g\rangle_m = \langle f,g\rangle_{m,\M}$ and $\|f\|_m=\| f\|_{m,\M}$. 
\end{definition}
When $\Omega$ is a region in Euclidean space $\reals^d$, the definitions above are equivalent to the standard ones (cf. \cite[Chapter III]{Adams-75-1}). In particular, when $\sfp=2$, the definitions coincide. Furthermore, if the $\sfp=2$, and the $\reals^d$ norms are expressed in the multi-index notation in (\ref{multi_index_superscript_notation}), then they have the form
\[
\|f\|_{W_\sfp^m\Omega)}^p = \sum_{|\alpha|\le m}\binom{|\alpha|}{\alpha} \|D^{\alpha}f\|_{L_\sfp(\Omega)}^\sfp.
\]

\subsection{Metric equivalences}\label{metric_equiv}
There is a  metric equivalence between the $\reals^d$-Sobolev norm on $\Omega \subset B(0,r)$, where $r<\inj$ and the $\M$-Sobolev norm on the image $\Exp_{p_0}(\Omega)\subset \b(p_0,r)$, $p_0\in \M$.   In particular, the following lemma shows that the exponential maps induce operators $f\mapsto f\circ \Exp_{p_0}$ that are boundedly invertible from all spaces $W_\sfp^m(\Exp_p(\Omega))$ to $W_\sfp^m(\Omega)$. This is especially important for us because the constants involved are dependent only on $r$, and not on $p_0$; that is, the equivalence between these spaces is independent of $p_0$.
\begin{lemma}\label{Fran}
For $m\in\nats$ and $0<r<\inj$,  there are constants $0 < c_1 <c_2$ so that
 for any measurable $\Omega \subset B_{r}$, for all $j\in \nats$, $j\le m$, and for any $p_0\in \M$, 
 the equivalence
$$
c_1\| u\circ\Exp_{p_0}\|_{W_\sfp^j(\Omega)}\le\|u\|_{W_\sfp^j(\Exp_{p_0}(\Omega))}\le c_2\| u\circ\Exp_{p_0}\|_{W_\sfp^j(\Omega)}
$$
holds for all $u:\mathrm{\Exp}_{p_0}(\Omega)\to \reals$. The constants $c_1$ and $c_2$ depend on $r$, $m$ and $\sfp$, but they are \emph{independent} of $\Omega$ and $p_0$.
\end{lemma}
\begin{proof}
For $\Omega \subset B(0,r)$ and $1\le \sfp<\infty$ we have
$$
\|u\|^\sfp_{j,\Exp_{p_0}(\Omega)} = \sum_{k=0}^m \int_{\Exp_{p_0}(\Omega)} | \nabla^k f |_{g,x}^\sfp \sqrt{\det\bigl(g_{ij}(x)\bigr)}\dif x.
$$
By Remark~\ref{det_dist_normal}, $\sqrt{\det\bigl(g_{ij}(x)\bigr)}$ is bounded above and below by constants depending only on $\mathrm r$. Similarly, by Corollary~\ref{chart_dependence}, the pointwise metric norms of the covariant derivatives $| \nabla^k f |_{g,p}$ are bounded uniformly by the corresponding Euclidean quantities. Combining these facts yields the result for $\sfp<\infty$. The $\sfp=\infty$ follows similarly.
\end{proof}

\subsection{Sobolev bounds for functions with scattered zeros in $\M$} \label{sobolev_zeros}
In this section, we show that Sobolev space functions having many zeros are uniformly small, provided the Sobolev space has high enough order. An immediate consequence of this gives 
pointwise error estimates for many kinds of interpolation processes. We will discuss this later, in Corollary~\ref{NSEE}.

\paragraph{The Euclidean case} We will need to discuss bounds on certain Sobolev norms for a special class of  domains in $\reals^d$. Following Brenner and Scott \cite[Chapter  
4]{Brenner-Scott-94-1}, we will say that a domain $\stardom$ is  \emph{star shaped with respect to a ball} $B(x_c,r)\subset \stardom $ if for every $x\in \stardom$, the closed  
convex hull of $\{x\}\cup B(x_c,r)$ is contained in $\stardom$.  If $\stardom$ is bounded, then there will be a ball $B(x_c,R)$ that contains $\stardom$. Of course, the diameter $d_\stardom$ of $\stardom$ satisfies $2r<d_\stardom <2R$. An important, useful geometric quantity associated with $\stardom$ is the \emph{chunkiness parameter} $\gamma$, which Brenner
and Scott \cite[Definition 4.2.16]{Brenner-Scott-94-1} define to be the ratio of $d_\stardom$
to the radius of the largest ball $B_\text{max}$ relative to which $\stardom$ is star shaped; i.e., $\gamma=d_\stardom/r_\text{max}$, where $r_\text{max}$ is the radius of $B_\text{max}$. It is easy to see that $\gamma \le \frac{2R}{r}$. In the case $\stardom=B(x_c,r)$, which is star shaped with respect to itself, the chunkiness parameter is $\gamma=2$. Another geometric property associated with $\stardom$ is a cone condition. Every  $x$ in $\stardom$ is the vertex of a cone $C_x \in \stardom$ whose axis is along $x_c-x$, radius is $r$, and aperture (maximum angle across the cone) is $\theta = 2\arcsin\left(\frac{r}{2R}\right)$ \cite[Proposition 2.1]{NWW}. 

\begin{lemma}\label{poly_bounds} 
Let $X=\{x_1,\ldots,x_N\} \subset \stardom$, let $\nats\ni k>0$, and let $h>0$. Suppose that $h=h_X$, the mesh norm of $X$ in $\stardom$, or, less stringently, that every ball $B(x,h)\subset \stardom$ contains at least one point  in $X$. If $h$ satisfies
\begin{equation}\label{h_cond}  
h\le \frac{\inrad \sin(\theta)}{4(1+\sin(\theta))k^2},  
\end{equation}  
and if $p\in\pi_k(\reals^d)$, then for every multi-index $\alpha$, $|\alpha | \le k$, then
\[  
\|D^\alpha p\|_{L_\infty(\stardom)}\le  
2\left(\frac{2k^2}{\inrad \sin(\theta)}\right)^{|\alpha|}  
\|p\|_{\ell_\infty(X)}.
\]  
\end{lemma}

\begin{proof}
Apply  \cite[Proposition~2.3]{NWW}, as modified by \cite[Remark~2.4]{NWW}.
\end{proof}

\begin{remark}\label{poly_bounds_remark}
\em Two things. First, if $h=h_X$, then every ball $B(x,h_X)\subset \stardom $ contains a point in $X$. Thus the ``or'' is not necessary. Second, if we take $r=r_\text{max}$ and $R=d_\stardom$, then $\frac{2R}{r} = \frac{2d_\stardom}{r_\text{max}}= 2\gamma$ and $\sin(\theta/2) =\frac{1}{2\gamma}$. It is easy to show that
\begin{equation}\label{replacement_bounds}  
\frac{d_\stardom}{16k^2\gamma^2}\le\frac{\inrad \sin(\theta)}{4(1+\sin(\theta))k^2} \ \text{\rm and}\ 
\frac{2k^2}{\inrad \sin(\theta)} \le \frac{4\gamma^2 k^2}{d_\stardom}.
\end{equation}
Thus the restriction (\ref{h_cond}) on $h$ and the bound on $\|D^\alpha p\|_{L_\infty(\stardom)}$ now become
\begin{equation}\label{h_cond2}  
h\le \frac{d_\stardom}{16k^2\gamma^2}\ \text{\rm and}\ 
\|D^\alpha p\|_{L_\infty(\stardom)} \le 2\left(\frac{4\gamma^2 k^2}{d_\stardom}\right)^{|\alpha|}\|p\|_{\ell_\infty(X)} .
\end{equation}
\em 
\end{remark}

\begin{proposition}\label{stardom_est}  
Let $\stardom\subset \reals^d$ be a bounded, star-shaped domain,  $m\in \nats$ and $\sfp\in \reals$, $1\le \sfp \le \infty$. Assume $m>d/\sfp$ when $\sfp>1$, and $m\ge d$, for $\sfp=1$. If $u\in W_\sfp^m(\stardom)$  satisfies $u|_X=0$, where $X=\{x_1,\ldots,x_N\}\subset \stardom$ and if $h=h_X\le  \frac{d_\stardom}{16k^2\gamma^2}$, then
\begin{equation}
\label{p_bound_W_k_u}
|u|_{W_\sfp^k(\stardom)} \le C_{m,d,\sfp} \gamma^{d+2k} d_\stardom^{m-k}  
|u|_{W_\sfp^m(\stardom)}.
\end{equation} 
\begin{equation}
\label{infty_bnd_u}
\|u\|_{L_\infty(\stardom)} \le C_{m,d,\sfp} \gamma^{d} d_\stardom^{m-d/\sfp} |u|_{W_\sfp^{m}(\stardom)}.
\end{equation} 
\end{proposition}  
\begin{proof}
Following the notation in Brenner-Scott \cite{Brenner-Scott-94-1}, we take $Q^m u\in \pi_{m-1}(\reals^d)$ to be the Taylor polynomial for $u$ averaged over $B_\text{max}$, the largest ball relative to which $\stardom$ is star shaped.  We begin by estimating $|Q^mu|_{W_\infty^k(\stardom)}$. Using $u|_X=0$, we have 
\[
\|Q^m u\|_{\ell_\infty(X)}=\|Q^m u - u\|_{\ell_\infty(X)}\le \|Q^m u - u\|_{L_\infty(\stardom)} .
\]
By \cite[Proposition~4.3.2]{Brenner-Scott-94-1}, with the $\gamma$-dependence involved there explicitly included, the right-hand side above has the bound 
\begin{equation}
\label{L_infinity_error}
\|Q^m u - u\|_{L_\infty(\stardom)}\le C_{m,d,\sfp}\gamma^d d_\stardom^{m-d/\sfp} |u|_{W_\sfp^{m}(\stardom)}.
\end{equation}
It follows that 
\[
\|Q^m u\|_{\ell_\infty(X)} \le C_{m,d,\sfp} \gamma^d d_\stardom^{m-d/\sfp} |u|_{W_\sfp^{m}(\stardom)}.
\] 
The next step requires applying Lemma~\ref{poly_bounds}, with  $r=r_{\text{max}}$ and $R=d_\stardom$, and then using Remark~\ref{poly_bounds_remark}. Doing so results in
\begin{equation}
\label{L-infinity_bnd_Q^m}
\|D^\alpha Q^m u\|_{L_\infty(\stardom)}\le  C'_{m,d,\sfp} \gamma^{d+2|\alpha|} d_\stardom^{m-| \alpha|-d/\sfp} |u|_{W_\sfp^{m}(\stardom)}.
\end{equation}
Sum over $|\alpha|=k\le m$ and use $\|D^\alpha Q^m u\|_{W_p^k(\stardom)} \lesssim d_\stardom^{1/\sfp}\|D^\alpha Q^m u\|_{L_\infty(\stardom)}$ to get
\begin{equation}
\label{W_p^k_bnd_Q^m}
|Q^mu|_{W_\sfp^k(\stardom)} \le C''_{m,d,\sfp}\gamma^{d+2k} d_\stardom^{m- k}|u|_{W_\sfp^{m}(\stardom)}.
\end{equation}
We need to estimate $|u-Q^mu|_{W_\sfp^{k}(\stardom)}$. To do this, we will use the Bramble-Hilbert Lemma \cite{Brenner-Scott-94-1}, which holds for $1\le p\le \infty$. Take $r=r_\text{max}$ in the lemma. We have, after carefully tracking the $\gamma$-dependence of the constant there, 
\begin{equation}
\label{error_sobolev_bnd}
|u-Q^mu|_{W_\sfp^{k}(\stardom)}\le C_{m,d}\gamma^d d_\stardom^{m-k} |u|_{W_\sfp^{m}(\stardom)}, \ k=0,1,\ldots,m.
\end{equation}
Using the triangle inequality in conjunction with (\ref{W_p^k_bnd_Q^m}) and (\ref{error_sobolev_bnd}) results in (\ref{p_bound_W_k_u}). In addition, doing the same with the bounds in  (\ref{L-infinity_bnd_Q^m}), for $|\alpha |$, and (\ref{L_infinity_error}) yields (\ref{infty_bnd_u}).
\end{proof}

The inequality (\ref{infty_bnd_u}) is a special case of the one established next.

\begin{corollary} With the notation of Proposition~\ref{stardom_est}, we have
\[
|u|_{W_\infty^k(\stardom)}\le C_{m,k,d,\sfp}\gamma^{2d+2(m-k)}d_\stardom^{m-k-d/\sfp} (1+d_\stardom^k)|u|_{W_p^m(\stardom)}.
\]
\end{corollary}
\begin{proof}
From our estimate on $|Q^mu|_{W_\infty^k(\stardom)} $, it is easy to show that 
\[
|Q^mu|_{W_p^k(\stardom)} \le C'_{m,d,\sfp}\gamma^{d+2k} d_\stardom^{m- k }|u|_{W_\sfp^{m}(\stardom)}.
\]
Putting the two estimates together via the triangle inequality yields (\ref{p_bound_W_k_u}). Now suppose that $u\in W_\infty^k(\stardom)$. $D^\alpha Q^m u=Q^{m-|\alpha|}D^\alpha u$, so
\[
\|D^\alpha (Q^m u - u)\|=\|Q^{m-|\alpha|}(D^\alpha u) - (D^\alpha u)\| \le C_{m-|\alpha|,d,\sfp} (1+\gamma)^d d_\stardom^{m-|\alpha|-d/\sfp} |u|_{W_\sfp^{m-|\alpha|}(\stardom)}.
\]
Hence, we have
\[
|Q^m u - u|_{W_\infty^k(\stardom)} \le C_{m-k,d,\sfp} \gamma^d d_\stardom^{m-k-d/\sfp} |u|_{W_\sfp^{m-k}(\stardom)}.
\]
By (\ref{p_bound_W_k_u}), $k$ replaced by $ m-k$, we finally arrive at
\[
|Q^m u - u|_{W_\infty^k(\stardom)} \le C_{m,k,d,\sfp}  \gamma^{2d+2(m-k)}d_\stardom^{m-d/\sfp} 
|u|_{W_p^m(\stardom)}.
\]
Combining this with our bound on $|Q^mu|_{W_\infty^k(\stardom)}$ and again employing the triangle inequality, we obtain the desired inequality.
\end{proof}

We will next apply the result above in the special case of a ball, where the bounds simplify considerably. Specifically, when $u\in W_p^m(B(x,r))$ and $u|_X=0$, it allows us to control certain sums of lower order Sobolev norms by $|u|_{W_p^m(B(x,r))}$. Doing this yields the following:

\begin{lemma}\label{genball} Let $X=\{x_1,\ldots,x_N\} \subset B(x,r)$ have its mesh norm $h=h(X,B(x,r))$ satisfy $h\le h_0r$, where $h_0:=\frac{1}{32m^2}$. In addition, suppose that $m>d/\sfp$, if $\sfp>1$, and that $m\ge d$, if $\sfp=1$. Then there is a constant $C_{m,d,\sfp}>0$ such that  the estimate
\begin{equation}\label{key_est_p}
\bigg(\sum_{k\le m} r^{\sfp(k-m)} | u|_{W_\sfp^k(B(x,r))}^2\bigg)^{1/\sfp} \le C_{m,d,\sfp} |u|_{W_\sfp^m(B(x,r))}
\end{equation}
holds for all $u\in W_\sfp^m(B(x,r))$ vanishing on $X$ (i.e., $u_{|_{X}}= 0$). In addition, we have that
\begin{equation}\label{infty_bnd_u_ball}
\|u\|_{L_\infty(B(x,r))} \le C_{m,d,\sfp}r^{m-d/\sfp} |u|_{W_\sfp^m(B(x,r))}
\end{equation}
\end{lemma}

\begin{proof}
For a ball $\stardom=B(x,r)$, which is of course star shaped, the chunkiness parameter is $\gamma=2$, the diameter $d_{B(x,r)} =2r$. Since $h\le \frac{d_\stardom}{16m^2\gamma^2}=h_0r$, Proposition~\ref{stardom_est}  applies. Thus, for $k=0,\ldots,m$, the bounds in (\ref{p_bound_W_k_u}) become 
\[
|u|_{W_\sfp^k(\stardom)} \le C_{m,d,\sfp} 2^{d+k} r^{m-k} |u|_{W_\sfp^k(\stardom)} 
\]
Standard algebraic manipulations of the expression above then yield (\ref{key_est_p}). The last inequality (\ref{infty_bnd_u_ball}) is a  direct  consequence of (\ref{infty_bnd_u}).
\end{proof}

\paragraph{The manifold case} The case of Sobolev bounds on $u$ when the underlying set  is a geodesic ball in $\M$ can be treated using a combination of the results involving metric equivalence, Lemma~\ref{Fran}, and the corresponding Sobolev bounds in Lemma~\ref{genball} for Euclidean balls. We can treat much more general situations than the one described below, but for now it is precisely what we need for the sequel.

\begin{lemma}[Zeros Lemma]\label{zeros}
Let $m$ be a positive integer, greater than $d/2$, and let $r$ be a positive real number less than $\inj$, the injectivity radius of $\M$. Suppose that $\Xi \subset \b(p,r)\subset \M$
is a discrete set with mesh norm $h \le \Gamma_1 r h_0$. If 
$u \in W^{m}_2 (\b(p,r))$ satisfies $u_{|_{\Xi}} = 0$, then for every $q\in \b(p,r)$,
$$|u(q)| \le C_{m,\M} r^{m-d/2}\|u\|_{W^{m}_2(\b(p,r))},$$ 
where $C_{m,\M}$ is a constant independent of $u,p,q,h$ and $r$, and where $\Gamma_1$ is as in (\ref{isometry}).
\end{lemma}
\begin{proof}
Using the diffeomorphism $\Exp_{p}$, 
we set $X = \Exp_{p}^{-1}(\Xi)$ and note that this set has
mesh norm $h(X,B(0,r))\le rh_0$ by (\ref{isometry}). Defining $\widetilde{u}$ as $u\circ \Exp_p$, 
and setting $q=\Exp_p(z)$,
we see that (\ref{infty_bnd_u_ball}) applies to $\widetilde{u}$, giving
$$u(q)= \widetilde{u}(z)\le C_{m,d}r^{m-d/2} \|\widetilde{u}\|_{W_2^m(B(0,r))}\le C_{m,\M}r^{m-d/2}  \|u\|_{W_2^m(\b(p,r))},$$
by Lemma \ref{Fran}.
\end{proof}

%
\subsection{The family of kernels $\kappa_{m,\M}$ on $\M$}\label{native_space_kernels}
We now are prepared to identify the family of kernels associated with bounded Lebesgue constants.
A well known fact, which also happens to be a simple consequence  of combining Lemma \ref{Fran} and the Sobolev embedding theorem on domains in $\reals^d$ via, is that $W_2^m(\M)$ is embedded in the space of continuous functions on $\M$, for $m>d/2$ \cite[\S 2.7]{Aub}. Consequently, point evaluation is a bounded linear functional, and $W_2^m$ has a unique reproducing kernel, which we define by $\kappa_{m,\M}:\M \times \M \to \reals$, although we often suppress the domain, writing $\kappa_m = \kappa_{m,\M}$. Being a reproducing kernel means that
\[
 f(x) = \langle f, \kappa_{m}(\cdot,x)\rangle_m 
\]
for all $f\in W_2^m$. The reproducing  kernel is necessarily \emph{strictly} positive definite: the formula $\sum_{\xi,\zeta \in \Xi} v_{\xi} v_{\zeta} \kappa_m(\xi,\zeta) = \|\sum_{\xi \in \Xi} v_\xi\kappa_m(\cdot, \xi)\|_m^2 =0$ implies that  there exist coefficients $(\alpha_{\zeta})_{\zeta\in \Xi}$ so that, for all $f\in W_2^m$, $f(\xi) = \sum_{\zeta \in \Xi\setminus\{\xi\}} \alpha_{\zeta}f(\zeta)$. Using a bump function centered at $\xi$ for $f$ easily provides a counterexample.

As an aside, we note that we can modify the Sobolev norms (and, hence, the reproducing
kernels) in the following benign way:
$$
\| f\|^2_{\mathcal{H}_m(\Omega)} := \sum_{k=0}^m C_k \int_{\Omega} |\nabla^k f|_{g,p}^2 \dif\mu(p),
$$
where $C_m>0$, $C_0>0$, and $C_k\ge 0$, $k=1,\ldots,m-1$. For such a modified norm, Lemma \ref{Fran} holds in precisely the same way, except with different
constants $c_1,c_2$. The benefit, when $\M=\reals^d$ is that for a particular choice of
constants $C_k$ we have the inner products corresponding to the Sobolev (or Mat\'{e}rn) splines
\cite{matern1986}, 
$$
\kappa(x,\alpha) = C_{m,d} |x|^{m-d/2} K_{d/2-m}(|x-\alpha|) 
$$
where $K$ is a modified Bessel function.
This is achieved when $\| f\|^2_{\mathcal{H}_m(\reals^d) } = \int_{\reals^d} f(x) (1-\Delta)^m f(x) \dif x,$ which is easily accomplished because
$\int_{\reals^d} f(x) \Delta^m f(x) \dif x =(-1)^m \int_{\reals^d} \langle \nabla^m f(x), \nabla^m f(x)\rangle \dif x$ holds by integration by parts.
%
%

\paragraph{Positive definite kernels and their native spaces} The situation above can be turned around, in the sense that we are able to start with a symmetric, \emph{strictly} positive definite kernel $\kappa$, and then construct a corresponding Hilbert space that is the RKHS for it. The term \emph{strictly positive definite} means that for any finite set $\Xi\subset \M$ the interpolation matrix $\calc_{\Xi} = \bigl(\kappa(\zeta,\xi)\bigr)_{(\zeta,\xi)\in \Xi^2}$ is positive definite.  The construction of the corresponding RKHS for such a kernel is described in detail in \cite[\S 10.2]{Wendland-05-1}. The space itself is known as the native space $\caln(\kappa)$, and its inner product is denoted by $\langle\cdot,\cdot\rangle_{\caln(\kappa)}$. Because $\calc_{\Xi} = \bigl(\kappa(\zeta,\xi)\bigr)_{(\zeta,\xi)\in \Xi^2}$ is positive definite, the interpolation problem 
$s_{|_{\Xi}} = f_{|_{\Xi}}$ for 
$s\in \spam_{\xi\in\Xi}{\kappa(\cdot,\xi)}$  
always possesses a unique solution, 
denoted by $I_{\Xi}f$.
Equivalently, $I_{\Xi}f$ can be determined 
by finding the solution to the variational problem 
$\mathrm{argmin}\{\|s\|_{\caln(\kappa)}: s\in \caln(\kappa),\  s_{|_{\Xi}} = f_{|_{\Xi}}\}.$ The relationship between positive definite kernels and their native space inner products translates into a duality between kernel interpolation and the variational problems with interpolatory constraints. This is discussed in detail in the celebrated Golomb-Weinberger paper \cite{Golomb-Weinberger-59-1}.

For a kernel associated with a radial basis function, the native space has its origin in the work of Madych and Nelson (see \cite{MaNe} for an example), where the kernel interpolation problem is recast as a variational problem, one where the interpolant is the minimizer of a Hilbert space norm (or seminorm) over all possible interpolants of data. The native space appellation itself is due to Schaback  \cite{Scha}, who extended this idea to treat kernels on more general domains.

We now turn to native space error estimates for interpolation by positive definite kernels associated with Sobolev native spaces. The following corollary to Lemma \ref{zeros} shows that, 
regardless of the kernel giving rise to a Sobolev native space,  it is always possible to measure the approximation order for interpolation from the native space. This generalizes the previous native space result for $\reals^d$ and $\sph^d$.

\begin{corollary}\label{NSEE}
Let $\kappa: \M \times \M \to \reals :(\eta,\zeta)\mapsto \kappa(\eta,\zeta)$ be a positive definite kernel with native space $\caln(\kappa) \cong W^{m}_2(\M)$, $m>d/2$. Assume that $\Xi\subset \M$ has mesh norm $h\le \inj h_0\Gamma_1$. Then for $f\in W^{m}_2(\M)$, the error incurred by interpolating with 
$\kappa$ at the nodes $\Xi \subset \M$ is
$$|f(x) - I_{\Xi}f(x)| \le C_{\M} h^{m-d/2}\|f\|_{W^{m}_2(\M)}.$$ 
\end{corollary}
\begin{proof}
By picking $r= h/(h_0 \Gamma_1)$, we have $r\le \inj$ and Lemma \ref{zeros} applies to the 
ball $\b(p,r)$, giving
$$|f(x) - I_{\Xi}f(x)|\le C_{\M}\left(\frac{h}{h_0\Gamma_1}\right)^{m-d/2}\|f-I_{\Xi}f\|_{W^{m}_2(\b(p,\frac{h}{h_0\Gamma_1}))} \le C_{\M}\left(h\right)^{m-d/2}\|f-I_{\Xi}f\|_{W^{m}_2(\M)}.$$
Since, by assumption $\|\cdot \|_{W^{m}_2(\M)} \sim \|\cdot\|_{\caln}$, we have the chain of inequalities
$\|f-I_{\Xi}f\|_{W^{m}_2(\M)}\le C \|f-I_{\Xi}f\|_{\caln}\le C\|f\|_{\caln} \le C\|f\|_{W^{m}_2(\M)},$
where the middle inequality follows by the Pythagorean theorem, since $I_{\Xi}$ is an orthogonal projector on the native space.
\end{proof}

\section{The Lagrange Function}\label{lagrange_function}
In this section, the Lagrange function centered at an arbitrary point $\xi\in \Xi$ 
(usually suppressing the subscript $\xi$: $\chi = \chi_{\xi}$) is investigated. 
We begin by showing that for a fairly general class of kernels, $\chi$
is bounded for quasiuniform centers.
In Section~\ref{lagrange_function_decay} it is shown shown that, for a specific class of kernels, 
$\chi$ is actually controlled by a rapidly
decaying function of $\d(x,\xi)/h$. Specifically, in the ball $\b(\xi,\inj)$ about $\xi$,  $|\chi(x)|$ is controlled by $\exp[-\nu \d(x,\xi)/h]$ (this is Proposition \ref{pointwise}).

\subsection{Uniformly Bounded Lagrange Functions}\label{lagrange_function_bnds}
Our first goal is to obtain bounds on the decay of the Lagrange (or fundamental) function
for kernel based interpolation. As before, we denote the native space for a positive
kernel by $\caln(\kappa)$. 
\begin{definition}\label{Lagrange}
Given a positive definite kernel $\kappa:\M^2\to \reals$,
and a finite set $\Xi\subset \M$
we denote the Lagrange function centered 
at $\xi\in \Xi$  
by $\chi_{\xi}$. I.e., $\chi_{\xi}(\zeta) = \delta(\xi,\zeta)$ for $\xi,\zeta\in \Xi$
and 
$\chi_{\xi} \in \spam_{\zeta \in\Xi}\kappa(\cdot,\zeta).$
By the discussion in Section 3, we see that 
$\chi_{\xi} = \mathrm{argmin} \{\|s\|_{\caln(\kappa)}: s(\zeta) = \delta(\xi,\zeta), \zeta\in \Xi\}.$
\end{definition}

We first observe that, when the centers are quasiuniform, the function $\chi$ is bounded.
\begin{lemma}\label{LagrangeBound}
Suppose that $\kappa$ is a positive definite kernel on $\M$ with native space $\caln(\kappa)= W^{m}_2(\M)$, $m>d/2$.
If the centers $\Xi$ are quasiuniform -- namely, there exists a constant $\rho $ such that $h/q\le \rho $, 
-- 
then the Lagrange function $\chi$ is
bounded, with a constant depending on $\kappa$, $m$ and $\rho $ only.
\end{lemma}
\begin{proof}
Let $\psi$ be a 'bump' function with support inside the ball 
$\b(\xi,q) = \{\alpha\in \man: d(\xi,\alpha)<q\}$ obtained by dilating a univariate function:
$$\psi(\alpha) = \sigma\left(\frac{\d(\xi,\alpha)}{q}\right)$$
where 
$\sigma\in C^{\infty}([0,\infty))$ is decreasing, has support in $[0,1)$, and 
satisfies $\sigma(0)=1$.  
The Lagrange function interpolates $\psi$, and Corollary \ref{NSEE} provides the
estimate:
$$\|\psi - \chi\|_{\infty} \le C h^{m-d/2} \|\psi\|_m.$$
From its definition, it is clear that $\psi$ is bounded: $\|\psi\|_{\infty}$ is $1$. Consequently,
$$\|\chi\|_{\infty} \le 1 + C h^{m-d/2} \|\psi\|_m.$$
Thus, we need only to estimate this last
quantity. Since $\psi$ is a function of the distance from $\xi$ only, we can use Lemma \ref{Fran} to estimate it. Note that, if $\Exp_{\xi}(x) = \alpha$ 
then
$\psi\circ \Exp_{\xi} (x)= \sigma\left(\frac{\d(\Exp_{\xi}(x),\Exp_{\xi}(0))}{q}\right) = \sigma\left(\frac{|x|}{q}\right)$.
So $\|\psi\|_m $ is controlled by 
$\|\sigma\left(\frac{|\cdot|}{q}\right)\|_m \le C_{m} q^{d/2-m}$
which means that
$$\|\chi\|_{\infty} \le 1 + C \rho ^{m-d/2}.$$
Here $C$ is the constant determined by the interpolation error, the constant from Lemma \ref{Fran} and  the $m^{\mathrm{th}}$ Sobolev norm of the univariate function $\sigma$. 
\end{proof}

\subsection{Rapid Decay of Lagrange Functions}\label{lagrange_function_decay}
%
For the  kernels $\kappa_{m,\M},$ we can improve significantly over Lemma \ref{LagrangeBound}, 
by employing an argument of Matveev \cite{Mat}, 
to obtain very rapid decay of the Lagrange function. The strong metric isomorphisms provided by (\ref{isometry}) and Lemma \ref{Fran} permit us, roughly, to ignore the manifold and carry out most of our analysis on the tangent plane. 

This has the additional benefit that, although it is not compact, one may take $\M = \reals^d$.  
In this case, the Riemannian distance is simply the Euclidean distance
$\d(x,y) = |x-y|$ and the exponential map at $\xi$ is nothing more than translation by $\xi$:
$\Exp_{\xi}(x) = x +\xi$.
Furthermore, the injectivity radius is $\inj=\infty$, the constants from (\ref{isometry}) are $\Gamma_1= \Gamma_2=1$, 
and, likewise, the constants from Lemma \ref{Fran} are $c_1= c_2=1$.
The underlying kernel is the one associated with the Sobolev norm on $\reals^d$
and is very similar to the so-called Sobolev (or Mat\'{e}rn) kernels. 
In fact, a benign modification to these arguments, namely, by replacing the Sobolev norm with either 
the (previously defined) Sobolev seminorm (in fact, this is the variational problem originally considered by Matveev \cite{Mat}) 
$$|f|^2_{m,\Omega} 
= \int_{\Omega} \langle \nabla^m f(x), \nabla^m f(x)\rangle \dif x$$ 
or with a slightly reweighted Sobolev norm
$$\| f\|^2_{\mathcal{H}_m(\Omega)} := \sum_{k=0}^m C_k \int_{\Omega} \langle \nabla^k f(x), \nabla^k f(x)\rangle \dif x,$$ 
as discussed in Section~\ref{native_space_kernels} gives totally equivalent results for the surface splines and the Sobolev splines, respectively.

The pointwise Lagrange function estimate is essentially a ``bulk chasing'' argument, adapted
from \cite[Lemma 5]{Mat}, where the central observation is that the ``bulk'' of the tail
of the Lagrange function's Sobolev norm is always contained in a narrow annulus.
This is a property shared by exponentially decaying functions: the integral of the tail of a nonnegative function of this type can be controlled by the integral on a sufficiently wide annulus. 
The next lemma makes this precise:
\begin{lemma}\label{mainprop}
Assume that $\Xi\subset \man$ has mesh norm $h< \Gamma_1 h_0 \min(1,\inj/3)$.
There exists a constant $\epsilon\in (0,1)$ (depending on $d$ and $m$) such that for $1\le t<\frac{\inj h_0 \Gamma_1}{3 h}$
\begin{equation}\label{E:mainprop}
\|\chi\|_{m,\comp\left(\xi,3t\Rad\right)} \le \epsilon \|\chi\|_{m,\comp\left(\xi,3(t-1)\Rad\right)}, 
\end{equation}
where $\Gamma_1$ is the constant from the left hand side of (\ref{isometry}), $\inj$ is the radius of injectivity and $h_0$ is the constant from Lemma \ref{genball}.
\end{lemma}
\begin{proof} We start by constructing a $C^{\infty}$ cutoff, $\phi$, that vanishes outside
$\b(\xi,(3t-1) \rad)$ and which equals $1$ on $\b(\xi,(3t-2)\rad)$. 
Let 
$$\phi(\alpha) = \sigma\left(\frac{\Gamma_1 h_0}{h}\d(\alpha,\xi)-3(t-1)\right)$$
where  $\sigma(T)$ equals $1$ for $T<1$ and vanishes for $T>2$. 
Since $\phi$ is a function of the distance from $\xi$ only, 
we can rewrite it as a composition of a univariate function and $\mathrm{Exp}_{\xi}$. Let  $\mathrm{\Exp}_{\xi}(x) = \alpha$, 
then
\begin{equation}\label{E:cutoff}
\phi\circ \mathrm{Exp}_{\xi} (x)= \sigma\left(\frac{\Gamma_1h_0}{h}\d(\Exp_{\xi}(x),\Exp_{\xi}(0)) -3(t-1)\right) = \sigma\left(\frac{\Gamma_1h_0}{h}|x |-3(t-1) \right).
\end{equation}

The region where $\phi$ is nonconstant -- $\a(\xi,3t-1,\rad)$ -- is 
 the middle third of the annulus  $\a(\xi,t,3\rad)$.
In order to simplify notation, we'll denote the annulus where $\phi$ transitions 
from $1$ to $0$ by the gothic letter $\aa$. That is,
$$\aa:= \a\left(\xi,3t-1,\rad\right).$$
Likewise, we'll denote the region where $\phi$ equals $1$ by
$$\bb:=\b\left(\xi,(3t-2)\rad\right),$$ 
and by $\bbcomp$ its complement $\man \backslash \bb$.

By the variational property of $\chi$, the Sobolev norm of $\chi$ is less 
than that of $\phi \chi$, since the two functions coincide on $\Xi$. 
Moreover, because of the compact support of $\phi$, the Sobolev norm of $\phi \chi$ can be
expressed in terms of its behavior on $\bb\cup \aa$, noting that on $\bb$ it is identical to $\chi$.
 Thus, we have that
$\|\chi\|_{m}^2\le 
\|\chi\|_{m,\mathfrak{b}}^2 + \|\phi \chi\|_{m,\aa}^2.$
By subtracting and applying Lemma \ref{Fran}, we obtain the estimate
\begin{equation}\label{TangentAnnulus}
\|\chi\|_{m,\bbcomp}^2\le \|\phi \chi\|_{m,\aa}^2\le c_2^{2}\|\ph\, \u\|_{m,\A}^2.
\end{equation}
where we identify corresponding annulus of interest in $\mathrm{T}_{\xi}M$ by
$$\A = A(0,3t-1,\rad).$$
The annulus $\A$ is the preimage, via the exponential map, of $\aa$,  $\u:= \chi\circ\Exp_{\xi}$, and $\ph = \phi\circ \Exp_{\xi}$.
\begin{center}
\psfrag{inner}{$\A = A\left(0,3t-1,\rad\right)$}
\psfrag{outer}{$A\left(0,t,3\rad\right) $}
\psfrag{Caption}{Figure 1: Important annuli on the tangent space $T_{\xi}M$. }
\includegraphics[height=4in]{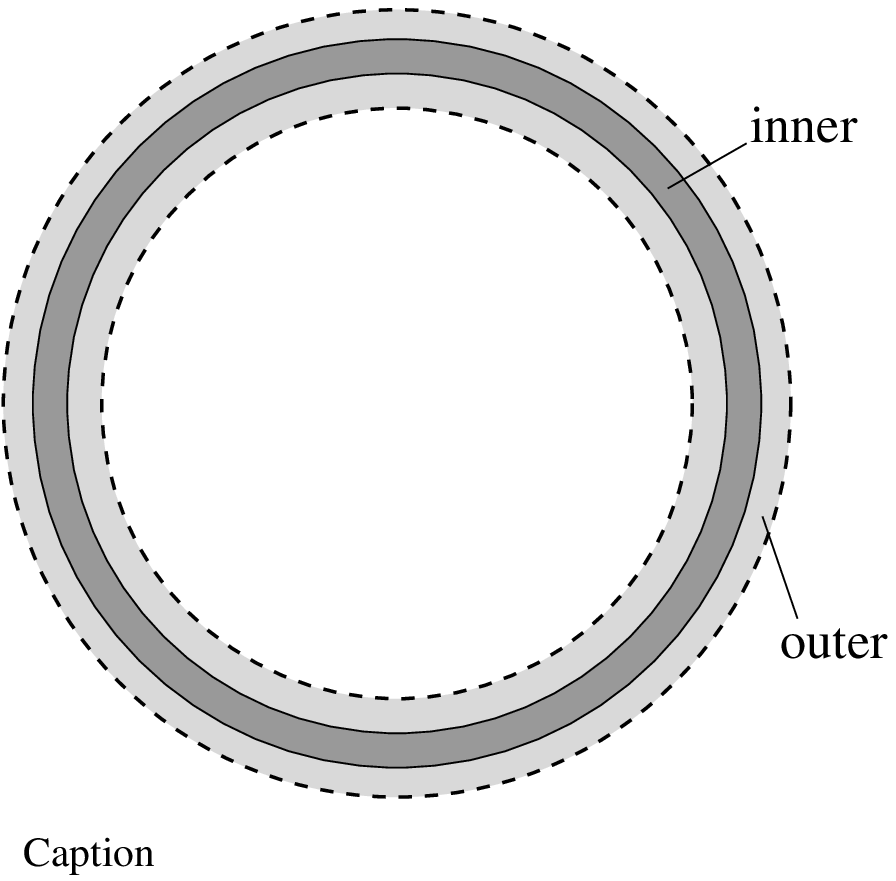}
\end{center}

Our goal for the rest of the proof is to estimate the quantity $\|\ph\, \u \|_{m,\A}^2.$
This involves two key estimates. The first estimate is concerned with the effect of multiplying by the cutoff (which we can tackle by using Leibniz' rule). The effect is independent of the radius $t$, and allows us to control this norm in terms of a combination of seminorms of $\u$. The second estimate will control this combination of seminorms via partitioning the annulus into small caps where Lemma \ref{genball} can be applied.

{\bf Estimate 1} The desired result is to control $\|\ph \, \,\u\|_{m,\A}^2$
by a biased Sobolev norm of $\u$ on $\A$  -- this is (\ref{E:productrule}).

We observe, first of all, that we can simplify matters by applying (\ref{E:cutoff}) and the product rule:
\begin{eqnarray*}
\|\ph \,\u\|_{m,\A}^2 
%
&=& \sum_{|\gamma| \le m} \binom{|\gamma|}{ \gamma} 
\int_{\A}  
  \left| 
    \sum_{\beta \le \gamma} 
      \binom{\gamma}{\beta}  D^{\gamma-\beta} \ph(x) \, D^{\beta}\u(x) 
  \right|^2\, 
\dif x\\
&\le& C_{m,d}\sum_{|\gamma| \le m} 
  \int_{\A} 
    \sum_{\beta \le \gamma} 
      \left|  
        D^{\gamma-\beta}
        \left[
          \sigma\left(\frac{\Gamma_1 h_0}{h}|x|-3(t-1) \right)
        \right]
      \right|^2 \, 
      \left|
        D^{\beta}\u(x) 
      \right|^2\, 
  \dif x\\
\end{eqnarray*}
A direct application of the inequality,
$|D^{\gamma}\sigma(r|x|-T)| \le C r^{|\gamma|}\|\sigma\|_{C^{(m)}}$, which holds for $T>0$, shows that $\|\ph \,\u\|_{m,\A}^2 $ is bounded by
$ C_{m,d} \sum_{|\gamma| \le m} 
  \int_{\A}  
    \sum_{\beta \le \gamma} 
        \left(\frac{\Gamma_1 h_0}{h}\right)^{2(|\gamma|-|\beta|)}
        \left|D^{\beta}\u (x)\right|^2
  \,  \dif x.$ 
Because $ \rad<1$, and 
$\left(\frac{\Gamma_1h_0}{h}\right)^{|\gamma|-|\beta|}
\le \left(\frac{\Gamma_1 h_0}{h}\right)^{m-|\beta|}$, 
we are left, after rearranging terms, with the estimate
\begin{equation}\label{E:productrule}
\|\ph \, \u\|_{m,\A}^2\le 
C_{m,d} \sum_{|\beta|\le m}  \left(\frac{\Gamma_1 h_0}{h}\right)^{2(m-|\beta|)}
\int_{\A}  \left|D^{\beta}\u (x)\right|^2 \, \dif x
\end{equation}
%
%
%

{\bf Estimate 2} We are now in a position to apply Lemma \ref{genball}. We cover $\A$ with a sequence of balls $(B_j)_{j\in \mathcal{J}}$ such that
\begin{itemize}
\item each ball is of radius $\rad$.
\item each ball is contained in the large annulus $A(0,t,3\rad)$ (displayed in Figure 1). 
\item every $x\in A(0,t,3\rad) $ 
is in at most $N(d)$ balls $B_j$, with $N(d)$ depending only on the spatial dimension and not on $t,h, h_0, \mathcal{J}$, etc.
\end{itemize}

Observe that the combination of seminorms from (\ref{E:productrule}) can be bounded by corresponding norms carried by the balls $B_j$. The mesh norm in each $B_j$ is $ h/\Gamma_1\le h_0$ (at most), so Lemma \ref{genball} applies, with 
$r= \rad$.
\begin{eqnarray*} 
\sum_{|\beta|\le m}
  \left(
    \rad
  \right)^{2(|\beta|-m)} 
  \int_{\A} 
    |D^{\beta} \u|^2 
&\le& 
\sum_{|\beta| \le m} 
  \sum_{j\in \mathcal{J}} 
    \left(
      \rad
    \right)^{2(|\beta|-m)} 
    \int_{B_j} 
      |D^{\beta} \u|^2\\
&\le& C_{m,\M} \sum_{j\in J}|\u|_{W_2^m(B_j)}^2\\
&\le& N(d) C_{m,\M} | \u|_{W_2^m\left(A(0,t,3\rad)\right)}^2
\end{eqnarray*}
The second inequality is Lemma \ref{genball}, while the third inequality is a consequence of the finite intersection property of the balls $B_j$. By a direct application of Lemma \ref{Fran}, we obtain
$$ 
\sum_{|\beta|\le m}
  \left(
    \rad
  \right)^{2(|\beta|-m)} 
  \int_{\A} 
    |D^{\beta} \u|^2 
\le c_1^{-2} N(d) C_{m,\M} \| \chi \|_{m,\a(\alpha,t,3\rad)}^2.$$

Putting this together with (\ref{E:productrule}) and (\ref{TangentAnnulus}), we see that the
Sobolev norm of the Lagrange function taken over the complement of a ball can be controlled
by the Sobolev norm on a thin annulus near its boundary:
\begin{eqnarray}
\|\chi\|_{m,\bbcomp}^2 
&\le& 
C_{m,\M} \|\ph\, \u \|_{m,\A}^2\nonumber\\
& \le& 
C_{m,\M} \|\chi\|_{m,\a\left(\xi,t,3\rad \right)}^2
=:
K \|\chi \|_{m,\a\left(\xi,t,3\rad\right)}^2
\label{sqrnorms}.
\end{eqnarray}
\begin{center}
\psfrag{C+}{$\bb^+$}
\psfrag{C-}{$\bb^-$}
\psfrag{annulus}{$\a(\xi,t,3\rad)$}
\psfrag{I}{$3(t-1)\rad$}
\psfrag{O}{$3t\rad$}
\psfrag{Caption}{Figure 2: Important sets on the manifold.}
\includegraphics[height=3in]{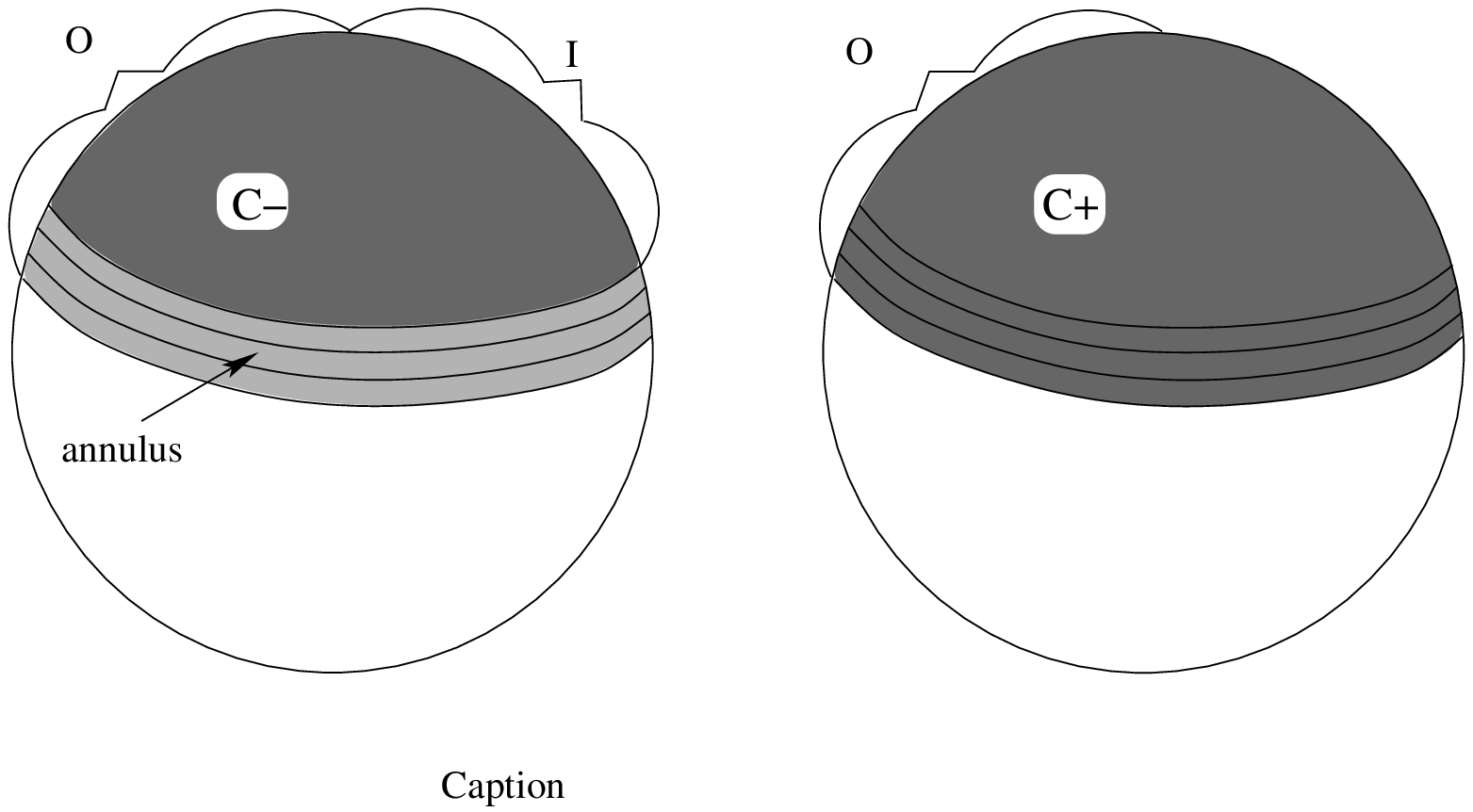}
\end{center}

What remains, is to reinterpret this as an inequality involving only complements of
concentric balls. 
It follows from their definitions that 
$\bb = 
\b\left(\xi, (3t-2)\rad\right)
\subset 
\b\left(\xi,3t\rad\right)
=:\bb^+$. 
This implies that $(\bb^+)^{\complement}\subset \bbcomp$ and,
setting 
$ \bb^-:= \b\left(\xi,3(t-1)\rad\right)$, 
we make the simple but useful observation that
$\a(\xi,t,3\rad) 
= 
(\bb^-)^{\complement}\setminus (\bb^+)^{\complement}.$ 
Applying this to (\ref{sqrnorms}) we get
$$
\|\chi\|_{m,(\bb^+)^{\complement}}^2\le 
K \left(\|\chi\|_{m,(\bb^-)^{\complement}}^2 - \|\chi\|_{m,(\bb^+)^{\complement}}^2\right),
$$
and, consequently:
$$
\|\chi\|_{m,\comp\left(\xi,3t\rad\right)}
\le 
\epsilon \|\chi\|_{m,\comp\left(\xi,3(t-1)\rad\right)}
$$
with $\epsilon:= \sqrt{\dfrac{K}{1+K}}<1$.
\end{proof}
For a sufficiently large value of $t$ less than a fixed multiple of the radius of injectivity, the previous lemma could be repeated several times, to obtain better estimates,
resulting in the following Corollary.
If we have $\M = \reals^d$, the radius of injectivity is $\inj=\infty$, and $t$ may become arbitrarily large. Furthermore, in the Euclidean case one may choose the surface splines as the kernel (instead of $\kappa_{\M,m}$), which results in \cite[Corollary p. 130]{Mat}. Alternatively, one may choose the Sobolev splines (see Section~\ref{native_space_kernels}) and achieve precisely 
the following:
\begin{corollary}\label{SobExpDec}
There is a constant $\nu>0$ such that, if $\Xi$ has mesh norm $h<\Gamma_1 h_0\min(1,\inj/3)$ then for $0\le T<\inj$ we have the estimate 
\begin{eqnarray*}
\|\chi\|_{m,\comp(\xi,T)} 
&\le& C_{m,\M} \, e^{-\nu(\frac{T}{h})} \|\chi\|_{m,\man}\\
&\le&  C_{m,\M} \, e^{-\nu(\frac{T}{h})} \, q^{d/2-m}
\end{eqnarray*}
where $q$ is the minimal separation distance between points of $\Xi$.
\end{corollary}
\begin{proof}
The corollary is a consequence of iterating Lemma \ref{mainprop} as often as possible. 
Let $T = 3t\rad= 3n\rad + r$ with $0\le r< 3\rad$.  
The integer $n$ represents the number of times the inequality (\ref{E:mainprop}) can be iterated. It follows that
\begin{eqnarray*}
\|\chi\|_{m,\comp(\xi,T)} 
%
&\le& 
\epsilon \|\chi\|_{m,\comp(\xi,3(t-1)\rad)}\\
&\le& 
\epsilon^n \|\chi\|_{m} \le \epsilon^{-1} \epsilon^t \;\|\chi\|_{m}\\
&=&
\epsilon^{-1} \left(\epsilon^{\frac{\Gamma_1 h_0}{3}}\right)^{\frac{T}{h}} \;\|\chi\|_{m}.
\end{eqnarray*} 
The first inequality in the statement of the corollary follows with $\nu = |\frac{\Gamma_1 h_0}{3}\log \epsilon|$ and $\epsilon^{-1}$ is absorbed into the constant.

The norm of $\chi$ can be estimated by using the technique at the end of the proof of Lemma \ref{LagrangeBound}. It is a consequence of the variational property of $\chi$,
namely that $\|\chi\|_{m}\le \|\psi\|_{m}$ with $\psi$ a properly scaled
bump function: $\psi(\alpha):=\sigma\left(\frac{\d(\xi,\alpha)}{q} \right)$. 
Computing its norm as we did in Lemma \ref{LagrangeBound} gives
$$\|\psi\|_m
\le
c_2\left\|\sigma\left(\frac{|\cdot|}{q}\right)\right\|_m \le C q^{d/2-m}$$
and the second inequality follows.
\end{proof}
We are now able to give the pointwise bound on the Lagrange function by applying Lemma \ref{zeros} (the zeros lemma). 
We note that when $\inj = \infty$, the first estimate in Proposition \ref{pointwise} holds throughout $\M$ (this is the case when $\M = \reals^d$). For the surface spline kernels
the following proposition holds by applying \cite[Corollary p.130]{Mat}, while for 
Sobolev splines it follows in precisely the same way as for $\kappa_{\M,m}$.
%
%
%
\begin{proposition}\label{pointwise}
If $\Xi$ has mesh norm $h<\Gamma_1h_0\min(1,\inj/2)$ then we have the estimate 
$$|\chi(\alpha)| \le C_{m,\M}\left(\frac{h}{q}\right)^{m-d/2}\,
\begin{cases}  
  \exp
  \left[-\nu\left(\frac{\d(\xi,\alpha)}{h}\right)\right]
  \quad &\text{for}\ 
  \d(\alpha,\xi)\le \inj\\
  \exp\left[-\nu\left(\frac{\inj}{h}\right)\right]
  \quad &\text{for}\ 
  \d(\alpha,\xi)> \inj.
\end{cases}
$$
\end{proposition}
\begin{proof}
For $\alpha$ near to $\xi$, say $\d(\xi,\alpha)\le \rad$, we can apply Lemma \ref{zeros} with $r=2\rad$
(i.e., on the ball $\b(\alpha,2\rad)$, where the zeros of $\chi_{\xi}$ have mesh norm bounded by $2h$)
to achieve
$$
|\chi(\alpha)| \le C_{m,\M}
\left(2\rad\right)^{m-d/2}
\|\chi\|_{m, \b(\alpha,2\rad)} 
\le
C_{m,\M}
\left(\frac{h}{q}\right)^{m-d/2}.$$
The second inequality is a consequence of Corollary \ref{SobExpDec} with $T=0$. Thus, for $\alpha \in \b(\xi,\rad)$ we have
$ |\chi(\alpha)| \le C_{m,\M}
\left(\frac{h}{q}\right)^{m-d/2} \exp(-\nu \frac{\d(\xi,\alpha)}{h})$.

%
A similar argument holds for $\rad<T=\d(\alpha,\xi)\le \inj$, where we can apply Lemma \ref{zeros} (now with 
$r= \rad$) to obtain
 \begin{eqnarray*}
  |\chi(\alpha)| 
&\le& C_{m,\M} \Rad^{m-d/2}  \|\chi\|_{W_2^m(\b(\alpha,\rad))}\\
&\le& C_{m,\M} h^{m-d/2}  \|\chi\|_{W_2^m(\b(x, T - \rad)^{\complement})}.
\end{eqnarray*}
The second inequality follows because the ball $\b(\alpha,\rad)$ is contained in $\b(\xi, T - \rad)^{\complement}$.
Applying Lemma \ref{SobExpDec} gives
 \begin{eqnarray*}
  |\chi(\alpha)| 
&\le& C_{m,\M} \left(\frac{h}{q}\right)^{m-d/2}\, \exp\left[-\nu\left(\frac{T - \rad}{h}\right)\right]\\
&\le& C_{m,\M} \left(\frac{h}{q}\right)^{m-d/2}\, \exp\left[-\nu\left(\frac{\d(\alpha,\xi)}{h}\right)\right].
\end{eqnarray*}
For $T>\inj$, we make use of the fact that $\comp(\alpha,\inj)\supset\comp(\alpha,T)$, so
we simply use the estimate for $T=\inj$ to obtain
$$|\chi(\alpha)| \le C \left(\frac{h}{q}\right)^{m-d/2}\, \exp\left[-\nu\left(\frac{\inj}{h}\right)\right],$$
since $\|\chi\|_{m,\comp(\alpha,T)} \le\|\chi\|_{m,\comp(\alpha,\inj)}.$
\end{proof}
\subsection{The Lebesgue Constant is Bounded}\label{lebesgue_const_bnd}
We are now in a position to prove our main theorem about the boundedness of the Lebesgue constant. 
It follows from the fast decay of the Lagrange functions by a standard argument that 
decomposes $\M$ \emph{en annuli}, counting the elements of $\Xi$ in each annulus
and balancing this against the influence of the far away Lagrange functions.
Quasiuniformity is used in two ways. First, to obtain bounds on Lagrange functions that
do not involve $h/q$, and, second, to estimate the cardinality of subsets of $\Xi$. We note
that $\Omega$, a measurable subset of $\M$, can be covered by a collection of small balls $\bigl(\b(\xi,q/2)\bigr)_{\xi\in \Xi\cap \Omega}$ that never overlap. Thus
\begin{equation}\label{card_est}
\# (\Xi\cap \Omega) \le \mu(\Omega)/\min_{\xi\in \Xi\cap\Omega}\mu(\b(\xi,q/2)\bigr)
\le C q^{-d}\mu(\Omega),
\end{equation}
with $C= C_{\M}$ depending only on the manifold.
\begin{theorem}[Lebesgue Constant]
Let $\M$ be a complete, compact Riemannian manifold of dimension $d$, and assume $m>d/2$. 
For a quasiuniform set $\Xi\subset \M$, with mesh ratio $h/q \le \rho $,
if $\tfrac{1}{2}h\le \Gamma_1 h_0\min(1,\inj/2)$, then the Lebesgue constant, 
$L = \sup_{\alpha\in \M} \sum_{\xi\in \Xi} |\chi_{\xi}(\alpha)|$,
associated with $\kappa_{m,\M},$
is bounded by a constant depending only on $m$, $\rho $ and $\M$.
\end{theorem}
\begin{proof}
Fix $x$.
We first split the sum 
$$\sum_{\xi\in \Xi} |\chi_{\xi}(\alpha)| = 
\sum_{\substack{\xi\in \Xi\\ \d(\xi,\alpha)\le \inj}} |\chi_{\xi}(\alpha)|+
\sum_{\substack{\xi\in \Xi\\ \d(\xi,\alpha)\ge \inj}} |\chi_{\xi}(\alpha)|=:I+II$$
into two parts. 

The ``outer'' sum, $II$, can be estimated by H{\"o}lder's inequality using (\ref{card_est})
and Proposition \ref{pointwise}:
$$II \le  C q^{-d}\mu(\M) \exp\left[-\nu\left(\frac{\inj}{h}\right)\right]
\le C \rho ^d \mu(\M) h^{-d} \exp\left[-\nu\left(\frac{\inj}{h}\right)\right].
$$
which is bounded indepent of $h$.

The inner sum, $I$, is further decomposed into sums over sets $\Xi_k:=\Xi \cap \a(\alpha,k,h)$. 
By (\ref{card_est}), the cardinality of $\Xi_k$ is less than 
$\#(\Xi\cap \b(\alpha,hk)\le C q^{-d} \mu(\b(\alpha,hk)) \le C (\rho  k)^d$. The Lagrange functions centered in 
$\Xi_k$ 
are bounded by 
$|\chi_{\xi}(\alpha)| 
\le C \exp\left[-\nu\left(k-1\right)\right]
\le C \exp\left[-\nu k\right]$, because
$\d(\xi,\alpha)\ge (k-1)h$. 
Thus $I$ can be estimated by
$$I\le C\sum_{k=1}^{\infty}\sum_{\xi\in \Xi_k} \exp\left[-\nu k\right] \le 
C \rho ^d \sum_{k=1}^{\infty}  k^d\exp\left[-\nu k\right],$$
which is also bounded.
\end{proof}

\begin{remark} 
{\em We note that, when $\M= \reals^d$, one may apply this theorem to the family of kernels
known as the surface splines, as well as to the family known as the Sobolev (or Mat\'{e}rn) splines. This follows in a straightforward way from Proposition \ref{pointwise}, which 
can easily be shown to hold for the Lagrange functions associated to either of these kernels.
The crucial difference is that, when $\M=\reals^d$, the radius of injectivity is $\inj=\infty$, and Lagrange function is bounded by  
$|\chi(\alpha)| \le C_{m,d}\left(\frac{h}{q}\right)^{m-d/2}\, 
  \exp
  \left[-\nu\left(\frac{|\xi - \alpha|}{h}\right)\right]$
throughout $\reals^d$. It follows that sum in the preceding proof 
is controlled by the inner sum $I$ only (and not $II$).}
\end{remark}

\section{Example: Interpolation on $\sph^2$}\label{example_S2}
In this section, we provide the example of the kernel $\kappa_{2,\sph^2}$ that gives
rise to bounded Lebesgue constants for quasi-uniform data and whose
$L^\infty$ approximation order can be given exactly. This will be
accomplished by establishing that $\kappa_2(x,y) =: \phi(x\cdot y)$ belongs to
a class of SBFs studied in \cite{MNPW} whose $L^p$, $1\le p\le \infty$,
approximation orders were explicitly obtained. For the readers
convenience, we review relevant details from \cite{MNPW} which contains a
full discussion of the material below.

Let $H_\ell$ denote the span of the spherical harmonics with fixed
order $\ell$ on $\sphere$.
The orthogonal projection $P_\ell$ onto $H_\ell$ is given by
\[
 P_\ell f = \sum^{N^d_\ell}_{m=1} \langle f,Y_{\ell,m}\rangle Y_{\ell,m}.
\]
Let $L_d := \sqrt{\lambda^2_d - \Delta_{\sphere}}$ be the
pseudo-differential operator where $\Delta_{\sphere}$ is the
Laplace--Beltrami operator on $\sphere$ and $\lambda_d = \frac{d-1}2$. The
Bessel potential Sobolev spaces have the norm given by
\[
\|f\|_{H^p_\gamma} := \left\|\sum^\infty_{\ell=0}
  (\ell+\lambda_d)^\gamma P_\ell f\right\|_{L^p(\sphere)}.
\]
Finally for $\beta=1,2,\ldots$, let $G_\beta$ be the Green's function
for $L^\beta_d$, i.e.,
\[
 L^\beta_d G_\beta(x\cdot y) = \delta_y(x).
\]
More generally let $\phi_\beta := G_\beta + G_\beta *\psi$ where
$\psi$ is an $L^1$ zonal function and let 
$S_{\Xi} :=
\spam \{\phi_\beta(x\cdot \xi)\colon \ \xi\in \Xi\}$ where $\Xi$ has
mesh ratio $ h/q \le \rho $.

The following was given in \cite[Theorem 6.8]{MNPW}.

\begin{theorem} Let $1\le p \le \infty,\enskip \gamma\ge 0,\enskip
\beta>\gamma + d/p' \quad$ $(1/p + 1/p'=1)$. If $f\in H^p_\beta$, then
with $\phi_\beta$ and $S_{\Xi}$ as given above, and for quasiuniform centers
$\Xi$ having mesh ratio $\rho $,
\[
\d_{H^p_\gamma}(f,S_{\Xi}) \le Ch^{\beta-\gamma}
\rho ^d\|f\|_{H^p_\beta}\quad (H^p_0 = L^p).
\]
\end{theorem}

We next wish to show that $\kappa_{2}(x,y)=\phi(x\cdot y)$ defined on $\sph^2$ of
the form $\phi = G_4 + G_4 *\psi$, $\psi \in L^1(\sph^2)$ and zonal. By
the above theorem, $\phi$ will have $L^\infty$ approximation order $4$ for
functions in $f\in C^4(\sph^2)$. For functions of lower smoothness,
a simple $K$-functional argument shows
that for $f \in C^{\beta}$,
$$\d_{\infty}(f,S_{\Xi}) \le Ch^{\beta} \|f\|_{C^{\beta}},$$
since such functions satisfy $K(f,t)\le C_{f,\beta} t^{\beta}$ 
for all $t>0$, where
$$K(f,t):=\sup_{g\in C^4(\sph^2)} \|f-g\|_{\infty} + t^4\|g\|_{C^4(\sph^2)}.$$

To this end let $\phi = \phi(x,y)$ be the reproducing kernel
for the Hilbert space with inner product
\[
\langle u,v\rangle_N := \langle \nabla^2u, \nabla^2v\rangle_{L^2(S^2)}
+ \langle\nabla u, \nabla v\rangle_{L^2(S^2)} + \langle u,v\rangle_{L^2(S^2)}
\]
where the $\nabla u$ is the covariant derivative and $\nabla^2u$, the
``Hessian,'' is given by
\[
\nabla^2f = \left(\frac{\partial^2f}{\partial x^i\partial x^j} -
  \Gamma^s_{i,j} \frac{\partial f}{\partial x^k}\right) e^i\otimes
e^j.
\]
Thus $\langle u,\phi\rangle_N = u(x)$ where
\begin{align*}
  \langle u,\phi\rangle_N &= \langle\nabla^2u, \nabla^2\phi\rangle +
  \langle\nabla u, \nabla \phi\rangle+ \langle u,
  \phi\rangle\\
  &= \langle u,L\phi\rangle -\langle u,\Delta \phi\rangle +
  \langle u,\phi\rangle\\
  &= \langle u,(L-\Delta+I)\phi\rangle =: \langle u,\widetilde
  L\phi\rangle.
\end{align*}

One next needs to identify $L$, which, in turn, will lead to the
identification of the Green's function $\phi$. This requires
calculating the Christoffel symbols.

We will do this in $(\theta,\varphi)$-spherical coordinates, with
$\theta$ being the colatitude and $\varphi$ is the azimuthal angle
(longitude). The orthonormal basis for a tangent plane at
$(\theta,\varphi)$ comprises unit tangent vectors along $\theta$ and
$\varphi$, $\{\bfe_\theta,\bfe_\varphi\}$. With some work, it can be shown that
\begin{align*}
\nabla^2 u = u_{\theta\theta}\, \bfe_\theta \otimes \bfe_\theta &+ \text{csc }
\theta(u_{\theta\varphi} - u_\varphi \cot \theta) (\bfe_\theta \otimes \bfe_\varphi +
\bfe_\varphi \otimes \bfe_\theta)  \\ 
&+\left(\frac1{\sin^2\theta} u_{\varphi\varphi} +
  \cot \theta u_\theta\right) \bfe_\varphi \otimes \bfe_\varphi.
\end{align*}
At this point we postulate that $\widetilde L$ is zonal in which case
all partial derivatives of $u$ with respect to $\varphi$ are zero. It
will be shown that $\widetilde L$ will have a zonal Green's
function. In this case, $\nabla^2u$ reduces to
\[
\nabla^2u = u_{\theta\theta}\, \bfe_\theta \otimes \bfe_\theta + \cot \theta
u_\theta \bfe_\varphi \otimes \bfe_\varphi
\]
and 
\[
\nabla^2 u \cdot \nabla^2v =u_{\theta\theta} v_{\theta\theta} +
\text{csc}^2 \theta\{\sin \theta \cos \theta u_\theta\} [v_{\varphi\varphi}
+ \cos \theta \sin \theta v_\theta].
\]
It follows that
\begin{align*}
\int\limits_{S^2} \nabla^2u \cdot \nabla^2v \ d\mu =& \int\limits_{S^2}
u_{\theta\theta} v_{\theta\theta} \sin \theta \ d\theta d\varphi \\ 
+\int\limits_{S^2} (\text{csc}^4 \theta  & \sin \theta \cos \theta
u_\theta) (v_{\varphi\varphi} + \sin \theta \cos \theta v_\theta) \sin
\theta\ d\theta d\varphi.
\end{align*}
Upon integration by parts first in the $\varphi$ variable and then in the
$\theta$ variable one obtains
\begin{align*}
  \int_{\sph^2} \nabla^2u\cdot \nabla^2v \ d\mu 
  &= \int_{\sph^2} (L_\theta u) v\sin \theta d\theta d\varphi \quad\text{where}\\
  L_\theta u &= \frac1{\sin\theta} \left\{\frac{d^2}{d\theta^2}
    \left(\sin\theta \frac{d^2u}{d\theta^2}\right) - \frac{d}{d\theta}
    \left(\cot \theta \cos\theta \frac{d\mu}{d\theta}\right)\right\}.
\end{align*}
Thus the Green's function $\phi$ must be a solution to
\begin{align}\label{eq1}
  \widetilde Lu &=  (L_\theta -\Delta +I) u \\
  &= \frac1{\sin\theta} (\sin\theta u'')'' - \frac1{\sin\theta} (\cos \theta\cot \theta u')' - \frac1{\sin\theta} (\sin\theta u')'+ u = 0  \nonumber
 \end{align}
for $\theta \in (0,\pi)$. 

The equation \eqref{eq1} has four analytic solutions in the interval
$(0,\pi)$. At $\theta=0$ and $\theta=\pi$, it has \emph{regular
  singular points}. We will use the method of Frobenius for analyzing
the solution and the associated singularities \cite[p.~132]{CodLev}. 

In this method, a solution to \eqref{eq1} is assumed to have a power
series solution of the form
\[
u(\theta) = \theta^\beta \sum^\infty_{k=0} b_k\theta^k, \text{ where
  $\beta\in \mathbb C$ is unknown.}
\]
Substituting $u(\theta)$ into \eqref{eq1} results in the indicial
equation, which is $\beta^2(\beta-2)^2=0$ in this case. Two solutions to \eqref{eq1} are determined by the two roots $\beta=0$ and $\beta=2$ and the corresponding equations for the $b_k$'s. The other two of the four linearly independent solutions
are then derived from these. Carrying out the details yields the four
solutions below.
\begin{align}
u_1(\theta) &= p_1(\theta^2) \text { and } u_2(\theta)=
\ln(\theta^2)p_2(\theta^2), \ \beta=0,\label{beta0}\\
u_3(\theta)&=\theta^2p_3(\theta^2)
\text{ and } u_4(\theta)=\theta^2\left( p_4(\theta^2) +
\ln(\theta^2)p_5(\theta^2)\right), \ \beta=2. \label{beta2}
\end{align}
The $p_j$'s are all power series convergent in a neighborhood of
$\theta=0$. In addition, they all satisfy $p_j(0)\ne 0$. 

It is important to note that these solutions have only \emph{even}
powers in their series. This is a consequence of the symmetry $\theta
\to -\theta$ of the differential equation and the method of Frobenius
itself.

There is another symmetry that preserves \eqref{eq1}; namely, $\theta
\to \pi - \theta$. This transforms the regular singular point $\theta
=\pi$ into the one at $\theta =0$. Consequently, the four functions
$u_j(\pi - \theta)$, $j=1$ to $4$, are linearly independent solutions
to \eqref{eq1} valid near $\theta = \pi$.

Thus the Green's function $\phi$ for $\widetilde L$ has the form
$\phi(\theta) = \sum\limits^4_{j=1} a_ju_j(\theta)$. Note that $\phi$
must have finite norm, i.e.,
\[
\|\phi\|^2_N := \langle\nabla^2\phi, \nabla^2\phi\rangle +\langle\nabla\phi, \nabla \phi\rangle+
\langle\phi,\phi\rangle.
\]
By Kondrakov's embedding theorem (see \cite{Aub}), $\phi$
necessarily must be continuous (as a function of $\theta$) on
$[0,\pi]$. Thus $\phi$ cannot have either a $u_2(\theta)$ component or
a $u_2(\pi - \theta)$ component. The singular behavior of the Green's function $\phi$ at $\theta=0$ is thus determined by $u_4(\theta)$. 

If  there is any singular behavior at $\theta=\pi$, it will come from $u_4(\pi -\theta)$. This behavior is precisely the same type as at $\theta=0$. If the Green's function $\phi$ at
$\pi$ had the same singular behavior as at $\theta=0$, then
$\widetilde L\phi$ would exhibit nonzero distributional behavior at
$\theta=\pi$, producing a second $\delta$ function there. Since this doesn't occur for the reproducing kernel, the Green's function $\phi(\theta)$ is analytic near $\theta=\pi$. Here is a summary of what's been found:
\begin{equation}
\label{summary_phi}
\phi(\theta) = \left\{
\begin{aligned}
 Ap_1(\theta^2)+ A\theta^2 p_3(\theta^2) &+ 
C\theta^2\left((p_4(\theta^2)+\ln(\theta^2)p_5(\theta^2\right)), \ \theta \approx 0,\\
 Dp_1((\pi - \theta)^2)&+ E(\pi - \theta)^2 p_3((\pi - \theta)^2) , \ \theta\approx \pi.
\end{aligned}
\right.
\end{equation}
We also mention that $\phi(\theta)$ is real analytic in $\theta$ on the interval $(0,\pi]$. (Of course, it is singular at $\theta=0$.)

There is one final step that we need to take to determine the order of approximation; namely, we need to obtain the Green's function $\phi(\theta)$ in terms of $t=\cos \theta$.  Near $\theta=0$, we write $\sin(\theta/2) = \sqrt{(1-t)/2}$ and obtain $\theta = 2\arcsin \sqrt{(1-t)/2}$, for which $\theta^2(t)$ is analytic in a neighborhood of $t=1$. At $\pi$, we use $\pi - \theta = 2\arcsin \sqrt{(1+t)/2}$. This gives us $(\pi - \theta(t))^2$ analytic near $t=-1$. Inspecting the structure of the solution in \eqref{summary_phi} shows that $\phi$ can be expressed in the following way as a function of $t=\cos \theta$:
\[
\phi(\theta(t)) =  \alpha_0(1-t)\ln(1-t) + \cdots +\alpha_K(1-t)^K\ln(1-t)+ f_K(t),
\]
with $\alpha_0\ne 0$, and where, by adjusting $K$,  $f_K$ can be taken to be as smooth as one wishes on $[-1,1]$. To finish up, observe that in \cite[p.~8]{MNPW}, this singularity corresponds to $\beta= 2\cdot 1+2=4$. Thus the approximation rate is $h^4$.
\bibliographystyle{siam}
\bibliography{HNW}

\begin{thebibliography}{10}

\bibitem{Adams-75-1}
{\sc R.~A. Adams}, {\em Sobolev spaces}, Academic Press [A subsidiary of
  Harcourt Brace Jovanovich, Publishers], New York-London, 1975.
\newblock Pure and Applied Mathematics, Vol. 65.

\bibitem{Aub}
{\sc T.~Aubin}, {\em Nonlinear analysis on manifolds. {M}onge-{A}mp\`ere
  equations}, vol.~252 of Grundlehren der Mathematischen Wissenschaften
  [Fundamental Principles of Mathematical Sciences], Springer-Verlag, New York,
  1982.

\bibitem{Brenner-Scott-94-1}
{\sc S.~C. Brenner and L.~R. Scott}, {\em The mathematical theory of finite
  element methods}, vol.~15 of Texts in Applied Mathematics, Springer, New
  York, third~ed., 2008.

\bibitem{cheeger_etal1982}
{\sc J.~Cheeger, M.~Gromov, and M.~Taylor}, {\em Finite propagation speed,
  kernel estimates for functions of the {L}aplace operator, and the geometry of
  complete {R}iemannian manifolds}, J. Differential Geom., 17 (1982),
  pp.~15--53.

\bibitem{CodLev}
{\sc E.~A. Coddington and N.~Levinson}, {\em Theory of ordinary differential
  equations}, McGraw-Hill Book Company, Inc., New York-Toronto-London, 1955.

\bibitem{Demko}
{\sc S.~Demko}, {\em Inverses of band matrices and local convergence of spline
  projections}, SIAM J. Numer. Anal., 14 (1977), pp.~616--619.

\bibitem{DMS}
{\sc S.~Demko, W.~F. Moss, and P.~W. Smith}, {\em Decay rates for inverses of
  band matrices}, Math. Comp., 43 (1984), pp.~491--499.

\bibitem{doC2}
{\sc M.~P. do~Carmo}, {\em Differential geometry of curves and surfaces},
  Prentice-Hall Inc., Englewood Cliffs, N.J., 1976.
\newblock Translated from the Portuguese.

\bibitem{doC1}
{\sc M.~P. do~Carmo}, {\em Riemannian geometry}, Mathematics: Theory \&
  Applications, Birkh\"auser Boston Inc., Boston, MA, 1992.
\newblock Translated from the second Portuguese edition by Francis Flaherty.

\bibitem{D2}
{\sc J.~Duchon}, {\em Splines minimizing rotation-invariant semi-norms in
  {S}obolev spaces}, in Constructive theory of functions of several variables
  (Proc. Conf., Math. Res. Inst., Oberwolfach, 1976), Springer, Berlin, 1977,
  pp.~85--100. Lecture Notes in Math., Vol. 571.

\bibitem{D1}
\leavevmode\vrule height 2pt depth -1.6pt width 23pt, {\em Sur l'erreur
  d'interpolation des fonctions de plusieurs variables par les
  {$D\sp{m}$}-splines}, RAIRO Anal. Num\'er., 12 (1978), pp.~325--334, vi.

\bibitem{eichhorn1991}
{\sc J.~Eichhorn}, {\em The boundedness of connection coefficients and their
  derivatives}, Math. Nachr., 152 (1991), pp.~145--158.

\bibitem{eichhorn2007}
\leavevmode\vrule height 2pt depth -1.6pt width 23pt, {\em Global analysis on
  open manifolds}, Nova Science Publishers Inc., New York, 2007.

\bibitem{Erb-etal-08-1}
{\sc W.~Erb and F.~Filbir}, {\em Approximation by positive definite functions
  on compact groups}, Numer. Funct. Anal. Optim., 29 (2008), pp.~1082--1107.

\bibitem{Golomb-Weinberger-59-1}
{\sc M.~Golomb and H.~F. Weinberger}, {\em Optimal approximation and error
  bounds}, in On Numerical Approximation, R.~F. Langer, ed., Madison, 1959,
  Univ. of Wisconsin Press, pp.~117--190.

\bibitem{Gutzmer-96-1}
{\sc T.~Gutzmer}, {\em Interpolation by positive definite functions on locally
  compact groups with application to {$\mbox {SO} (3)$}}, Result.\ Math., 29
  (1996), pp.~69--77.

\bibitem{Hang}
{\sc T.~Hangelbroek}, {\em Polyharmonic approximation on the sphere},
  Submitted,  (2009).

\bibitem{hicks1965}
{\sc N.~J. Hicks}, {\em Notes on differential geometry}, Van Nostrand
  Mathematical Studies, No. 3, D. Van Nostrand Co., Inc., Princeton,
  N.J.-Toronto-London, 1965.

\bibitem{MJ}
{\sc M.~J. Johnson}, {\em The {$\mathit{L}\sb p$}-approximation order of
  surface spline interpolation for {$1\leq p\leq 2$}}, Constr. Approx., 20
  (2004), pp.~303--324.

\bibitem{jones-etal-2008}
{\sc P.~W. Jones, M.~Maggioni, and R.~Schul}, {\em Manifold parametrizations by
  eigenfunctions of the {L}aplacian and heat kernels}, Proc. Natl. Acad. Sci.
  USA, 105 (2008), pp.~1803--1808.

\bibitem{MaNe}
{\sc W.~R. Madych and S.~A. Nelson}, {\em Multivariate interpolation and
  conditionally positive definite functions. {II}}, Math. Comp., 54 (1990),
  pp.~211--230.

\bibitem{matern1986}
{\sc B.~Mat{\'e}rn}, {\em Spatial variation}, vol.~36 of Lecture Notes in
  Statistics, Springer-Verlag, Berlin, second~ed., 1986.
\newblock With a Swedish summary.

\bibitem{Mat}
{\sc O.~V. Matveev}, {\em Spline interpolation of functions of several
  variables and bases in {S}obolev spaces}, Trudy Mat. Inst. Steklov., 198
  (1992), pp.~125--152.

\bibitem{MNPW}
{\sc H.~N. Mhaskar, F.~J. Narcowich, J.~Prestin, and J.~D. Ward}, {\em
  $\mathit{L}^p$ {B}ernstein estimates and approximation by spherical basis
  functions}, Math. Comp., to appear,  (2009).

\bibitem{NWW}
{\sc F.~J. Narcowich, J.~D. Ward, and H.~Wendland}, {\em Sobolev bounds on
  functions with scattered zeros, with applications to radial basis function
  surface fitting}, Math. Comp., 74 (2005), pp.~743--763.

\bibitem{Reimer}
{\sc M.~Reimer}, {\em Hyperinterpolation on the sphere at the minimal
  projection order}, J. Approx. Theory, 104 (2000), pp.~272--286.

\bibitem{Scha}
{\sc R.~Schaback}, {\em Native {H}ilbert spaces for radial basis functions.
  {I}}, in New developments in approximation theory ({D}ortmund, 1998),
  vol.~132 of Internat. Ser. Numer. Math., Birkh\"auser, Basel, 1999,
  pp.~255--282.

\bibitem{triebel1992}
{\sc H.~Triebel}, {\em Theory of function spaces. {II}}, vol.~84 of Monographs
  in Mathematics, Birkh\"auser Verlag, Basel, 1992.

\bibitem{Wendland-05-1}
{\sc H.~Wendland}, {\em Scattered data approximation}, vol.~17 of Cambridge
  Monographs on Applied and Computational Mathematics, Cambridge University
  Press, Cambridge, 2005.

\bibitem{Castel-etal-07-1}
{\sc W.~zu~Castell and F.~Filbir}, {\em Strictly positive definite functions on
  generalized motion groups}, in Algorithms for approximation, Springer,
  Berlin, 2007, pp.~349--357.

\end{thebibliography}
\end{document}